\documentclass[11pt]{article}
\usepackage{amsmath}
\usepackage{amsthm}
\usepackage{amssymb}
\usepackage{latexsym}
\usepackage{lscape}
\usepackage{epsfig}
\usepackage{pstricks}
\usepackage{enumerate}
\usepackage{exscale}
\usepackage{relsize}
\usepackage{amsmath}
\usepackage{indentfirst}
\usepackage{graphicx} %use graph format
\usepackage{epstopdf}
\usepackage{lineno}
\usepackage{subcaption}
\usepackage{authblk}	
\usepackage{graphicx}
\usepackage[square,sort,comma,numbers]{natbib}

\newtheorem{theorem}{Theorem}[section]
\newtheorem{lemma}[theorem]{Lemma}
\newtheorem{proposition}[theorem]{Proposition}
\newtheorem{corollary}[theorem]{Corollary}
\theoremstyle{definition}
\newtheorem{definition}{Definition}
\newtheorem{example}[theorem]{Example}

\theoremstyle{remark}
\newtheorem{remark}{Remark}

\usepackage{mathrsfs,graphicx,latexsym,tikz,color,euscript}
\usepackage{amsfonts,amssymb,amsmath,amscd,amsthm}
\usepackage{epstopdf}
\usepackage{CJKutf8}
\usepackage[colorlinks=true,linkcolor=blue,urlcolor=blue,citecolor=blue,linkbordercolor=white,urlbordercolor=white,citebordercolor=white]{hyperref}
\usepackage{color}

\newcommand{\beq}{\begin{eqnarray*}}
	\newcommand{\eeq}{\end{eqnarray*}}

\theoremstyle{plain}

\theoremstyle{definition}

\begin{document}

	\title{Birkhoff sections in $3$-manifold with invariant toric foliation}
	
	\author{
        Wentian Kuang$^{1,}$\thanks{Partially supported by National Natural Science Foundation of China(No.12271300, No.11901279). E-mail: kuangwt@gbu.edu.cn}\\
		$^{1}${\small School of Sciences, Great Bay University,
			Dongguan, 523000, P.R. China,\\Great Bay Institute for Advanced Study, Dongguan 523000, China.\\}
	}
	
	\date{}
	
	\maketitle
	
	\begin{abstract}
		In this paper, we study the Birkhoff sections in a $3$-manifold foliated by invariant tori. We establish the necessary and sufficient conditions for various types of periodic orbits to serve as boundary orbits of a Birkhoff section. The construction relies on the dynamical behaviour of the flow combined with fundamental topological argument. As an application,  we study the boundaries of toric domains and the energy hypersurfaces of separable Hamiltonian systems, providing conditions for the existence or non-existence of different types of Birkhoff sections. Additionally, we offer an alternative proof of part of the results presented in \cite{vanKoert2022} and \cite{Hutchings2022}.
	\end{abstract}

\section{Introduction}
\subsection{History and some related results.}
     The notion of a Birkhoff section(see Definition \ref{def:birkhoffsection}) is a classical tool for studying flows on $3$-manifolds, originating from the works of Poincaré and Birkhoff. A global surface of section is a special case of a Birkhoff section that is embedded in the manifold. Such sections enable the reduction of the dynamics of a vector field on a $3$-manifold to the study of the return map on the section. In this framework, periodic orbits of the flow correspond to periodic points of the return map on the Birkhoff section.
     
     Another advantage of a Birkhoff section is that it reveals certain favorable properties of the flow. When a Birkhoff section exists, the manifold can be foliated by an $S^1$ family of Birkhoff sections, giving rise to a rational open book decomposition. Moreover, if the Birkhoff section is a global surface of section, this yields an open book decomposition of the manifold..
     
     A classical result of Birkhoff \cite{Birkhoff1966} demonstrated that the geodesic flow on a positively curved Riemannian $2$-sphere admits an annulus-like Birkhoff section. In  \cite{Fried1983},  Fried showed that any transitive Anosov vector field on a closed 3-manifold admits a Birkhoff section. In the last few decades, the study of Birkhoff sections and global surface of sections have attracted considerable interest.  It should be noted that not every flow on a closed 3-manifold admits Birkhoff sections. It is necessary that the flow must have periodic orbits, which does not always hold by the counterexample constructed by Kuperburg \cite{Kuperburg1996}.  In \cite{Taubes2007}, Taubes proved the Weinstein conjecture in dimension $3$, establishing that every Reeb vector field on a closed 3-manifold possesses a periodic orbit. This naturally raises the question of whether every Reeb vector field admits a Birkhoff section, as conjectured in  \cite{Colin2023}.
     
     The work of Giroux \cite{Giroux2002} implies that every contact structure can be supported by a contact form whose Reeb vector field admits a global surface of section. However, for a given Reeb vector field, determining whether a Birkhoff section exists is typically a challenging problem.
     
     In the celebrated paper \cite{Hofer1998}, Hofer, Wysocki, and Zehnder demonstrated that the canonical Reeb flow of any convex $3$-sphere embedded in $\mathbb{R}^4$ admits a global surface of section diffeomorphic to a disk. Combining this result with Franks' theorem \cite{Franks1992}, they showed that such a Reeb flow must have either two or infinitely many closed Reeb orbits, which is a special case of Hofer-Wysocki-Zehnder's "two or infinity" conjecture. This conjecture have been proven under various hypothesis, see \cite{Hofer2003,VPUA2024,Dan2019,Colin2023}. Recently, in \cite{Liu2024}, the authors proved the conjecture in cases where the associated contact structure has torsion first Chern class, with the existence of a Birkhoff section playing a crucial role in their argument.
     
    By combining the theory of embedded contact homology \cite{Hutchings2014} with Fried’s techniques in \cite{Fried1983}, Colin, Dehornoy and Rechtman \cite{Colin2023} proved that every non-degenerate contact form on $3$-manifold is carried by a broken book decomposition, which generalized the concept of an open book decomposition. Based on this, it is proved that a $C^{\infty}$-generic contact form on a closed $3$-manifold admits a Birkhoff section \cite{VPUA2024,CM2022}. It is conjectured in \cite{Colin2023} that any Reeb flow admits a Birkhoff section. This conjecture, if true, would further imply that every Reeb flow has either two or infinitely many closed orbits.
     
     Although the existence of a Birkhoff section is ensured in many situations, the precise shape of the Birkhoff section for a given flow often remains unknown. The ideal scenario is that the Birkhoff section becomes a global surface of section with a simple topology. It is proved in \cite{Hofer1998} that a dynamically convex contact form on $S^3$ admits a disk-like global surface of section, which is the simplest type of Birkhoff section. In \cite{HS2011}, Hryniewicz and Salom\~ao provided necessary and sufficient conditions for a Reeb orbit on tight $S^3$ to bound a disk-like global surface of section. In \cite{Hryniewicz2014}, it's proved that any Hopf orbit(i.e. the orbit is unknotted and has self-linking number $-1$.) for a dynamically convex contact form on $S^3$ bounds a disk-like global surface of section. In \cite{HSW2023}, the authors exhibit sufficient conditions for a finite collection of periodic orbits of a Reeb flow on a closed 3-manifold to bound a positive global surface of section. Most of these results rely on the theory of pseudo‐holomorphic curves. Readers can refer to \cite{Pedro2022} for a survey on recent developments in transverse foliations for Reeb flows via pseudo-holomorphic curves.
     
     Note that the boundary orbit of a disk-like global surface of section must be linked with all other Reeb orbits. In \cite{vanKoert2020}, van Koert constructed a Reeb flow on $S^3$ that does not admit a disk-like global surface of section by contradiction to the linking condition. In \cite{vanKoert2022}, the authors extended these results to show the non-existence of global surfaces of section with fewer than $n$ boundary components.
     
     Recently, Edtmair \cite{Edtmair2024} prove that the cylinder capacity of a dynamical convex domains in $\mathbb{R}^4$ agree with the least symplectic area of a disk-like surface of section for the Reeb flow on the boundary of the domain. This result establishes a connection between global surfaces of section and symplectic capacities. Note that there are many different symplectic capacities. The strong Viterbo conjecture states that all normalized capacities coincide for convex domains in $\mathbb{R}^4$. Since every normalized capacity $c$ satisfy $c_B\le c\le c_Z$, where $c_B$ is the ball capacity(also known as the Gromov width) and $c_Z$ is the cylinder capacity, $c_B=c_Z$ implies the strong Viterbo conjecture. However, Haim-Kislev and Ostrover \cite{HO2024} recently  constructed convex bodies in $\mathbb{R}^{2n}$, for which the strong Viterbo conjecture fails. In \cite{Hutchings2022}, Gutt and Hutchings proved the strong Viterbo conjecture holds for dynamically convex toric domains in $\mathbb{R}^4$.  Due to \cite{ZhangJun2024}, a dynamically convex toric domain needs not be convex. It's mysterious to find suitable conditions under which the strong Viterbo conjecture holds.
     
     The aim of this paper is to study Birkhoff sections for the flows that admit a regular invariant toric foliation(Definition \ref{def:toricfoliation}). As an application, we investigate the properties of Birkhoff sections on the boundaries of toric domains and on regular energy hypersurfaces of separable Hamiltonian systems.  Our construction provide an alternative proof of part of the results established in \cite{vanKoert2022} and \cite{Hutchings2022}.
     
\subsection{Settings and main results}
Throughout this paper, all vector fields are assumed to be sufficiently smooth. Let $M$ be an oriented $3$-manifold equipped with a non-vanishing vector field $X$. Denote $\phi^t$ to be the corresponding flow generated by $X$.
\begin{definition}\label{def:birkhoffsection}
	A Birkhoff section for the flow of $X$ is an immersion $\iota: S\rightarrow M$ defined on a compact surface $S$ such that:
	\begin{itemize}
		\item[(1)] if $\partial S\ne \emptyset$, then $\iota(\partial S)$ consists of periodic orbits of $X$;
		\item[(2)] $\iota^{-1}(\iota(\partial S))=\partial S$ and $\iota: S\backslash\partial S\rightarrow M\backslash \iota(\partial S)$ defines an embedding transverse to $X$; 
		\item[(3)] for every $p\in M$, there exist $t_-<0<t_+$ such that $\phi^{t_\pm}(p)\in \iota(S)$.
	\end{itemize}
	
\end{definition}

If $\iota$ is simultaneously an embedding and a Birkhoff section, then $\iota(S)$ is called a global surface of section for the flow of $X$. When there is no ambiguity, we often denote $S\subset M$ as a Birkhoff section or a global surface of section, instead of $\iota:S\rightarrow M$. 

In this paper, we study Birkhoff sections, or global surfaces of section, for integrable systems. The flow under consideration satisfies the following conditions:

\begin{definition}\label{def:toricfoliation}
	Let $M$ be an oriented $3$-manifold with a non-vanishing vector field $X$. $M$ admits a regular invariant toric foliation if 
	\begin{itemize}
		\item[(1)] There is an open dense set $O\subset M$ foliated by invariant tori, such that the flow on each torus is conjugate to a linear motion,
		\item[(2)] Every connected component of $M\backslash O$ is either a periodic orbit  or a regular broken torus(see Definition \ref{def:brokentorus}).
	\end{itemize}
\end{definition}

The regular energy hypersurfaces for Liouville integrable Hamiltonian systems certainly satisfy this condition.
Also, the flow on the boundary of a toric domain is foliated by invariant tori. It is worth noting that the theory of pseudo‐holomorphic curves usually requires the Reeb flow to be non-degenerate. However, this non-degeneracy condition is not satisfied for an integrable system, as periodic orbits are not isolated on a rational invariant torus.

Note that if a Birkhoff section intersects an invariant torus transversally, the intersection curve should be a closed transverse curve on the invariant torus. Our strategy is to construct transverse curves on all invariant tori such that these curves change smoothly as invariant tori vary. The union of all these curves will form a Birkhoff section. Usually, these curves cannot always be made transverse to the flow. When a curve fails to be transverse, it corresponds to a boundary orbit of the Birkhoff section. A crucial point in our construction is the following fundamental observation: \textbf{the homology type of transverse curves on a family of invariant tori remains constant(Lemma \ref{lem:invarianthomology}).}

To define the homology type of a transverse curve on a torus, we must first choose two generators of the first homology group of a torus. Additionally, an orientation for the transverse curve is required. It easy to see that an orientation of an invariant torus extends to all tori in $M$. By definition of Birkhoff section, we can assume that these transverse curves on the tori are oriented in such a way that the flow is positively transverse to the curve.

A torus can be viewed as a rectangle with opposite edges identified,a fact that will be used frequently throughout this paper. The two natural generators of the first homology group are the horizontal and vertical lines. A closed curve is called $(p,q)$-type if the corresponding homology type is $(p,q)$. In general, there may be different families of invariant tori in $M$, and the natural generators for each family might not coincide. There could be a non-trivial correspondence between the generators of different families, making it not always straightforward to define the homology types of closed curves across all invariant tori. However, as we will discuss in Section 3, this does not pose any issues for our construction.  For simplicity, we focus on the case where there is a consistent way to define homology types of closed curves on all invariant tori(i.e. the natural generators of different families of tori coincide). 

Our results include a series of lemmas and propositions. Below is a summary of the results concerning the conditions for the existence of boundary orbits of a Birkhoff section.
\begin{theorem}[Theorem \ref{thm:main}]\label{thm:main_intro}
	If $M$ admits a regular toric invariant foliation and the natural generators on different families of tori coincide, then there is a consistent way to define the homology of closed curves on invariant tori in $M$. Let $S$ be a Birkhoff section in $M$, we have
	\begin{itemize}
		\item[(1)] If a family of tori ends in an invariant vertical line $P_v$ with transverse curves of $(p,q)$-type. Then $\partial S\cap P_v=0$ if and only if $q=0$. Otherwise, the boundary curve is a $|q|$-fold covering of the periodic orbit;
		\item[(2)] If a family of tori ends in an invariant horizon line $P_h$ with transverse curves of $(p,q)$-type. Then $\partial S\cap P_h=0$ if and only if $p=0$. Otherwise, the boundary curve is a $|p|$-fold covering of the periodic orbit;
		\item[(3)] Let $T_0$ be an invariant tori of $(P,Q)$-type and $(p^+,q^+),(p^-,q^-)$ are corresponding types of transverse curves in two nearby families of tori. Then $\partial S\cap T_0\ne \emptyset$ if and only if $(p^+,q^+)$ is not a multiple of $(P,Q)$ and $(p^+,q^+)-(p^-,q^-)=k(P,Q)$ for some $k\ne 0$. The boundary curve will be $|k|$-covering of the periodic orbit.
		\item[(4)] If $B$ is a broken torus with invariant vertical lines, then $B\cap \partial S= \emptyset$ if and only if 
		\begin{equation}
			p_i^-=p_j^+\ne 0, \quad \sum_{i=1}^k q_i^-=\sum_{j=1}^{l}q_j^+.
		\end{equation}
		\item[(5)] If $B$ is a broken torus with invariant horizon lines, then $B\cap \partial S= \emptyset$ if and only if 
		\begin{equation}
			\sum_{i=1}^k p_i^-=\sum_{j=1}^{l}p_j^+, \quad 	q_i^-=q_j^+\ne 0.
		\end{equation}
	\end{itemize}
where $(p_i^-,q_i^-),(p_j^+,q_j^+)$ are homology types of transverse curves on different families of tori in neighbourhood of $B$.
\end{theorem}

By Theorem \ref{thm:main_intro} and related lemmas, we can construct various Birkhoff sections if sufficient information is available about the dynamics on different families of tori.  This approach can also be applied to construct examples of Birkhoff sections of different types. 

\begin{corollary}
	If $M$ is an oriented $3$-manifold with a non-vanishing smooth vector field $X$ and $M$ admits a regular invariant toric foliation with finitely many broken tori, then the flow admits a Birkhoff section.
\end{corollary}
\begin{proof}
	We begin with any invariant torus and choose a possible homology type for a transverse curve on it. We then extend this transverse curve continuously to nearby tori, ensuring that its homology type remains consistent across the family. If the chosen homology type fails to remain transverse across the entire family of tori, we apply part $(3)$ of Theorem \ref{thm:main_intro} to introduce an additional boundary orbit of the Birkhoff section when necessary. When encountering a broken torus, we use parts $(4)$ and $(5)$ of the theorem to determine the homology type of transverse curves on the nearby tori. If a family of tori terminates at a periodic orbit, $(1)$ and $(2)$ of the theorem determine whether this periodic orbit serves as a boundary orbit of the Birkhoff section.
	
	By our assumption, this process concludes in a finite number of steps. Thus, we successfully construct a Birkhoff section on $M$.
\end{proof}

Theorem \ref{thm:main_intro} can be applied to construct Birkhoff sections when $M$ admits a regular invariant toric foliation. Usually, the Birkhoff section can have non-zero genus. In \cite{Dehornoy2022}, when $M$ is a rational homology three-sphere, the authors provide a formula to compute the Euler characteristic of a transverse section in terms of the linking number and self-linking number of its boundary orbits. The genus of the section can then be derived from the Euler characteristic. In our setting, since the construction is explicit, both the Euler characteristic and genus of the Birkhoff section can be computed directly. See section \ref{sec:application} for details.

There can be various Birkhoff sections for a flow on $M$. In fact, by the resolution procedure as showed in \cite{Colin2023}, one can produce different Birkhoff sections. The question is to find the "simplest" Birkhoff section a flow can have. The simplest case is a disk-like global surface, which is guaranteed in certain situations \cite{Hofer1998,HS2011}. On the other hand, there exist Reeb flows on $S^3$ for which the Birkhoff section must have many boundary orbits. It is not easy to determine the lower bound for the number of boundary orbits among all possible Birkhoff sections.

In our setting, due to Corollary \ref{cor:necefordisk}, the flow is very restricted when it admits a Birkhoff section with a unique boundary orbit. This allows us to easily determine whether a periodic orbit bounds a Birkhoff section. In the application section, we provide a simple criterion for the existence of a Birkhoff section with a unique boundary orbit, both for the boundaries of toric domains and for separable Hamiltonian systems. Note that we do not need either the flow to be dynamically convex or a priori the topology of $M$. 

Obviously, the existence of a disk-like global surface of section implies $M$ to be isomorphic to $S^3$. In both cases discussed in the application section, if the flow admits a Birkhoff section with a unique boundary orbit, it actually admits a disk-like global surface of section. We do not know if this phenomenon is general. It is an interesting problem to find a Reeb flow on $S^3$ such that it admits a Birkhoff section with a unique boundary orbit but does not admit a disk-like global surface of section.

\subsection{Applications}
Recall that the flow on the boundary of a toric domain, as well as the flow on a regular energy level set of separable Hamiltonian $H=H_1(x_1,y_1)+H_2(x_2,y_2)$, admits a toric foliation. Our construction can be applied to such manifolds to derive some characterizations of their Birkhoff sections. 

\subsubsection{Boundaries of toric domains}
The definition of a toric domian is given in Definition \ref{def:toric} through the map
\[\mu:\mathbb{C}^2\rightarrow \mathbb{R}^2, \quad (z_1,z_2)\rightarrow (\pi|z_1|^2,\pi|z_2|^2).\]

Let $f: [0,1]\rightarrow \mathbb{R}^2$ be a smooth curve in first quadrant such that $f(0)=(0,a)$ and $f(1)=(b,0)$ for some $a,b>0$. Denote $M_f$ as the $3$-manifold given by $\mu^{-1}(f)$, which corresponds to the boundary of a toric domain. $M_f$ consists of a single family of invariant tori and carries a standard flow. For each $s\in (0,1)$, the preimage $\mu^{-1}(f(s))$ is an invariant torus. We prove the following theorem regarding disk-like or annulus-like Birkhoff sections in $M_f$.

\begin{theorem}[Theorem \ref{thm:birktoric}]\label{thm:birktoric_intro}
	The flow on $M_f$ has the following properties:
	\begin{itemize}
		\item[(1)] $\mu^{-1}(f(0))$ bounds a disk-like Birkhoff section if and only if $(0,1)$ is not a normal vector at $f(s)$ for all $s\in(0,1)$; 
		\item[(2)] $\mu^{-1}(f(1))$ bounds a disk-like Birkhoff section if and only if $(1,0)$ is not a normal vector at $f(s)$ for all $s\in(0,1)$; 
		\item[(3)] $\mu^{-1}(f(0))$ and $\mu^{-1}(f(1))$ bound an annulus-like Birkhoff section if and only if there exist $p\ne 0$ and $q\ne0$ such that $(p,q)$ is not a normal vector at $f(s)$ for all $s\in(0,1)$;
		\item[(4)] a $(p,q)$-type periodic orbit on a rational invariant torus $\mu^{-1}(f(s_0))(0<s_0<1)$ bounds a Birkhoff section if and only if $(1,0)$ is not a normal vector to $f(s)$ for all $s\in(0,s_0]$ and $(0,1)$ is not a normal vector to $f(s)$ for all $s\in[s_0,1)$. Moreover, if the boundary is $k$-fold covered, the genus of the Birkhoff section is given by $\frac{k(p-1)(q-1)}{2}$ .
	\end{itemize}
\end{theorem}

For global surface of sections in $M_f$, we have a similar result, which is stated in Theorem~\ref{thm:globaltoric}. The detailed statement is omitted here for brevity.

\begin{corollary}
	The boundary of a weakly convex toric domain admits a disk-like global surface of section.
\end{corollary}

\begin{remark}
	Theorem \ref{thm:globaltoric} provides necessary and sufficient conditions for $M_f$ to admit a disk-like Birkhoff section, under assumptions weaker than the toric domain being weakly convex. Notably, in \cite{Hutchings2022} it was shown that a toric domain is dynamical convex if and only if $f$ satisfies a strictly monotone condition. By \cite{Hofer2003}, it admits a disk-like global surface of section. As a result of Theorem \ref{thm:birktoric_intro}, we conclude that if $M_f$ is dynamical convex, then any periodic orbit bounds a Birkhoff section.
\end{remark}

By Proposition \ref{pro:boundary_exist}, we provide a much simpler construction of the following result from \cite{vanKoert2022}.
\begin{theorem}[[Theorem 1.1 in \cite{vanKoert2022}]]
	For any integer $k>0$, there exist a star-shaped $M_f$ with a flow equivalent to a Reeb flow on $S^3$, such that any Birkhoff section for the flow must have at least $k$ distinct boundary orbits.
\end{theorem}

Combining Theorem \ref{thm:birktoric_intro} with the result of Edtair \cite{Edtmair2024}, we obtain an alternative proof of the following result by Gutt and Hucthings.
\begin{theorem}[Theorem 1.7 in \cite{Hutchings2022}]
	If $X_\Omega$ is a monotone toric domain in $\mathbb{R}^4$, then $c_B(X_\Omega)=c_Z(X_\Omega)$.
\end{theorem}

Here $c_B$ and $c_Z$ are normalized symplectic capacities, defined as follows.
\begin{equation}
	\begin{split}
		c_B(X_\Omega)=sup \Big\{r\big|B(r)\stackrel{s}\hookrightarrow X_\Omega \Big\};\\
		c_Z(X_\Omega)=inf \Big\{r\big|X_\Omega \stackrel{s}\hookrightarrow Z(r) \Big\}.
	\end{split}
\end{equation}
where $"\stackrel{s}\hookrightarrow"$ represent a symplectic embedding.

\subsubsection{Separable Hamiltonian}
A Hamiltonian system is said to be \textsl{separable} if the Hamiltonian function can expressed as the sum of several independent two-dimensional Hamiltonian functions, i.e. $H=\sum_{i=1}^{n}H_i(x_i,y_i)$. In this paper, we focus on the case when $n=2$, thus each regular energy level set of $H$ will be a $3$-manifold with a non-vanishing Hamiltonian vector field. It is obvious that every regular energy level set of a separable Hamiltonian $H=H_1(x_1,y_1)+H_2(x_2,y_2)$ admits a toric invariant foliation. Consequently, our results can be directly applied to study Birkhoff sections for such flows.

We assume the following conditions: (i) both $H_1$ and $H_2$ has a unique global minimum point with minimum value being $0$; (ii) all critical points of $H_1$ and $H_2$ are isolated. Given a regular level set $H=c$, the resulting three-manifold is foliated by the product space $\{H_1=s\}\times \{H_2=c-s\}$ with $s\in [0,c]$.

It is straightforward to observe that there exists a consistent way to define the homology of transverse curves across all invariant tori in $H=c$. Furthermore, the vector of the linear flow on each torus has both components positive. A direct result of our construction is that the flow on any regular level set of $H$ admits a Birkhoff section. 

We also provide a necessary and sufficient condition for the existence of a disk-like global surface of section, which can be easily verified.

\begin{theorem}[Theorem \ref{thm:oneboundaryforH}]\label{thm:oneboundaryforH_intro}
	Let $H=c$ be a regular level set. Then the following are equivalent:
	\begin{itemize}
		\item[(i)] $H=c$ admits a disk-like global surface of section;
		\item[(ii)] $H=c$ admits a Birkhoff section with a unique boundary orbit;
		\item[(iii)]  Both sublevel sets $H_1^c$ and $H_2^c$ are isomorphic to a disk. And there exists a $c_0\in (0,c)$ such that there are no critical points of $H_1$ for $H_1>c_0$ and no critical points of $H_2$ for $H_2>c-c_0$.
	\end{itemize}
\end{theorem}

Clearly, the topology of the level set $H=c$ depends on $c$. So does the possible Birkhoff section in $H=c$. In the final section of this paper, we present an example(Example \ref{exm:5.10}) that illustrates an interesting phenomenon: as $c$ increases, the possible Birkhoff section in $H=c$ first becomes more complex and then simplifies again.

The paper is organized as follows. In Section \ref{sec:2}, we study the possible transverse curves on a invariant torus and within a family of tori. Section \ref{sec:3} focuses on the transverse curves near a broken torus. In Section \ref{sec:4}, we provide the necessary and sufficient conditions for existence of a boundary orbit of the Birkhoff section. Our main result Theorem \ref{thm:main_intro}, serves as a summary of the findings in Section \ref{sec:4}. Finally, in Section \ref{sec:application}, we apply our construction to the boundaries of toric domains and separable Hamiltonian systems, providing some criterion results and examples.

\section{Transverse curves on a family of invariant tori}\label{sec:2}
\begin{definition}
	For an invariant torus $T=S^1\times S^1$, let $g_1$ and $g_2$ be the two generators of the first homology of $T$. An oriented closed curve in $T$ is called $(p,q) $-type if its homology class is $(p,q)\in \mathbb{Z}\oplus\mathbb{Z}$.
\end{definition}
Denote $\mathbf{Ins}(C_1,C_2)$ as intersection number of two closed curves $C_1$ and $C_2$ on $T$, which is a topological invariant depending on the homology type of $C_1$ and $C_2$. For any closed curve $C\subset T$, it's clear that $C$ is $(p,q)$-type if and only if $\mathbf{Ins}(C,g_2)=p$ and $\mathbf{Ins}(g_1,C)=q$. 

For convenience, we usually regard $T$ as a rectangle with the opposite edges identified, a representation that will be used frequently throughout the paper. The vertical and horizontal lines correspond to the two generators $g_1$ and $g_2$. We also assume that the flow on any invariant torus is a uniformly linear motion.

The following two lemmas can be directly derived from the definition of transverse curve.

\begin{lemma}\label{lem:pqtransversecurve}
	There is no transverse curve of $(0,0)$-type on an invariant torus. An invariant torus admits a $(p,q)$-type transverse curve if and only if $(p,q)$ is not a tangent vector of the linear motion on the torus. 
\end{lemma}
\begin{proof}
	If a curve is of $(0,0)$-type, it is a contractable curve with all direction of tangent vectors. There must exist a point where the curve is tangent to the flow line. Thus there is no transverse curve of $(0,0)$-type.
	
	Now, consider the case when $(p,q)\ne (0,0)$. Since the flow is linear. If $(p,q)$ is not a tangent vector of the flow, then the line through $(0,0)\in T$ with tangent vector $(p,q)$ is always transverse to the flow. And it corresponds to a closed curve of $(p,q)$-type. 
	
	If the torus has a $(p,q)$-type transverse curve. By the mean value theorem, there exist some point such $(p,q)$ is a tangent vector of the curve. Thus $(p,q)$ cannot be a tangent vector of the linear motion on the torus. 
\end{proof}

\begin{example}
	Consider $H=|z_1|^2+|z_2|^2$. Then $H=1$ is the standard $S^3$. One easily verifies that the following surfaces are different types of global section.\\
	(i) $r_1\in (0,1], \theta_1=0$. This is an open disk with boundary orbit $|z_2|=1$.\\
	(ii) $(r_1,\theta,\sqrt{1-r_1^2},-\theta)$, where $r_1\in(0,1), \theta\in S^1$. This is an open annulus with two boundary orbits $|z_1|=1$ and $|z_2|=1$. 
\end{example}
Actually, it is possible to construct an annulus-like Birkhoff section for standard shpere such that the two boundaries are a $p$-fold cover of the $|z_1|=1$ and a $q$-fold cover of $|z_2|=1$. To achieve this, one just need to ensure that the intersection of the Birkhoff section with every invariant torus is a $(p,-q)$-type transverse curve.

\begin{lemma}\label{lem:mntype}
	If $\Gamma=S\cap T$ is a transverse curve of $(m,n)$-type. Then $\Gamma$ is composed of $l=\mathbf{gcd}(m,n)$ distinct connected components, where each component is of  $(\frac{m}{l},\frac{n}{l})$-type. Here $\mathbf{gcd}(m,n)$ is the greatest common divisor of $m$ and $n$(if either $m=0$ or $n=0$, then $l$ is absolute value of the non-zero integer).
\end{lemma}
\begin{proof}
	Since $S$ is embedded, $\Gamma$ is a 1-dimensional closed manifold. In rectangle coordinates, $\Gamma$ intersects a vertical line $n$-times and a horizon line $m$-times. We can deform the curve within the homology class so that these intersection points are uniformly distributed at the boundaries of rectangle. Because $\Gamma$ cannot self-intersect, it follows that $\Gamma$ is isotopic to a set of uniformly distributed parallel lines in the rectangle, with tangent vector $(m,n)$. Hence, $\Gamma$ is composed of $l$ components with each component to be a curve of $(\frac{m}{l},\frac{n}{l})$-type.
\end{proof}

\begin{corollary}\label{cor:homotopy}
	The space of transverse curves on a torus with a given homology type is contractible.
\end{corollary}

\begin{lemma}\label{lem:invarianthomology}
	Let $T_s$ be a family of tori intersecting $S$ transversely. Denote $\Gamma_s=T_s\cap S$. Then the homology of $\Gamma_s$ does not depend on $s$.
\end{lemma}
\begin{proof}
	By Lemma \ref{lem:mntype}, a transverse curve on an invariant torus decomposes into several disjoint closed curves, each with the same homology type. By transversality, each connected component induces a unique family of closed transverse curves that continuously depend on $s$, and different families do not intersect. Thus the number of connected components of $\Gamma_s$ remains constant. 
	
	It suffice to prove that the homology of each connected component of $\Gamma_s$ does not depend on $s$. Without loss of generality, assume $\Gamma_s$ consists of a single component. In this case, $\Gamma_s$ defines an embedding: $S^1\rightarrow T_s$. The homology of $\Gamma_s$ is determined by the degree of the following map:
	
	\[S^1\xrightarrow{\ \Gamma_s\ }T_s\xrightarrow{\ pr \ } S^1.\]
	where $pr: T_s=S^1\times S^1\rightarrow S^1$ is either of the two canonical projection maps. Since $\Gamma_s$ is continuous with respect to $s$, the degree must be constant. 
\end{proof}

To construct a Birkhoff section $S$, it suffices to ensure that the intersection of $S\backslash\partial S$ with every invariant torus of the flow is a transverse curve. The idea is to construct transverse curves on each invariant torus, allowing them to vary continuously with the tori. The closure of the union of these transverse curves will then form the desired Birkhoff section. If, in addition, the Birkhoff section is embedded at the boundary (i.e., the boundary orbits are simply covered), the resulting surface is a global surface of section.

Lemma \ref{lem:invarianthomology} reveals that certain constraints arise when constructing these transverse curves globally, which will play a crucial role in proving the non-existence of certain types of Birkhoff sections. On the other hand, Lemma \ref{lem:pqtransversecurve} illustrates the flexibility of transverse curves on an individual invariant torus, enabling the construction of various types of Birkhoff sections.

By combining Lemma \ref{lem:invarianthomology} and Lemma \ref{lem:pqtransversecurve}, we obtain the following proposition.

\begin{proposition}\label{pro:boundary_exist}
	Let $T_s, s\in (a,b)$ be a family of invariant tori such that linear flow covers all directions on the plane, i.e. the unit tangent vector of the flow go through all points in $\mathbb{RP}^1$. Then $\partial S\cap \bigcup_{s\in (a,b)} T_s\ne \emptyset$.
\end{proposition}

\begin{proposition}\label{pro:nopq}
	Let $T_s, s\in (a,b)$ be a family of invariant tori. Suppose there exist a unique $s_0\in (a,b)$ such that $\partial S\cap T_{s_0}\ne \emptyset$. Assume the periodic orbits on $T_{s_0}$ are of $(p,q)$-type. Then the homology type of transverse curves on $T_s$ for $s\ne s_0$ cannot be a multiple of $(p,q)$.
\end{proposition}
\begin{proof}
	For a non-transversal torus $T_{s_0}$, the intersection $\partial S\cap T_{s_0}$ consists of boundary orbits and some transversal segments. Since the transversal part extends to nearby tori, the limit curve of transverse curves on nearby tori must contain these transversal segments.
	
	We only prove the case for $s\in (a,s_0)$.	Assume the transverse curve on $T_s$ is of $k(p,q)$-type($k\ne 0$). If $\Gamma_s=S\cap T_s$ does not converge to a curve of constant slope as $s\rightarrow s_0^-$. Then the slopes of the transverse curves will contain an interval $(\frac{p}{q}-\epsilon,\frac{p}{q}+\epsilon)$ for some $\epsilon>0$. On the other hand, for $\epsilon>0$ small enough, the slopes of motion on $T_s(s\in (a,s_0))$ have non-empty intersection with $(\frac{p}{q}-\epsilon,\frac{p}{q}+\epsilon)$. Contradicts to the transverse assumption. 
	
	Thus $\Gamma_{s_0}=\lim_{s\rightarrow s_0} \Gamma_s$ must be a curve of constant slope $\frac{p}{q}$. $\Gamma_{s_0}$ must be composed of some periodic orbit, which cannot be transversal to flow lines in $T_{s_0}\backslash \partial S$. Our assumption that $T_s$ is of $k(p,q)$-type($k\ne 0$) does not hold. 
\end{proof}

The two propositions above show that the dynamical behavior on a family of invariant tori imposes certain restrictions on possible types of Birkhoff sections. In fact, Proposition \ref{pro:boundary_exist} is sufficient to construct a Reeb flow on $S^3$ with non-simple Birkhoff section. More specifically, for any positive integer $k$, there exist a Reeb flow on $S^3$ such that corresponding Birkhoff sections must have at least $k$ different boundary orbits. This result will be given as an application in Section \ref{sec:application}.

\section{Transverse curves near a transversal broken torus}\label{sec:3}
In Section \ref{sec:2}, we studied the behavior of transverse curves on a family of invariant tori. Typically, $M$ contains different families of invariant tori. After removing these regular invariant tori, the remaining part consists of singular invariant sets. These singular invariant sets may either be periodic orbits that serve as the ends of families of tori, or they could be invariant sets that separate distinct families of tori. For convenience, we introduce the notion of a broken torus. Throughout this section, $B$ will always denote a regular broken torus, and 
$S$ will refer to a Birkhoff section.

In the neighborhood of a broken torus $B$, different families of tori may exhibit distinct types of transverse curves. Since these families are separated by $B$, each family asymptotically approaches part of the broken torus. By continuity, the transverse curves on each family of invariant tori will converge to part of the intersection curves on $B$. As a result, the transverse curve on a broken torus will completely determine the homology types of the transverse curves on nearby invariant tori. Understanding the possible types of transverse curves that a broken torus can support, as well as the relationships between the homology classes of transverse curves across adjacent families of tori, is therefore essential. In this section, we explore these questions in detail.

\begin{definition}\label{def:brokentorus}
	A regular \textsl{broken torus} is a connected, closed invariant set that separates families of invariant tori. It consists of a finite collection of isolated periodic orbits and open annuli, where the boundary of each annulus is formed by two of these periodic orbits.
\end{definition}

\begin{remark}
	By definition, certain undesirable cases are excluded, such as: (i) singular invariant sets that contain a continuous family of periodic orbits; (ii) singular invariant set that occupy an open subset of $M$. 
\end{remark}

\begin{lemma}\label{lem:rectanglerepre}
	$B$ can be represented as a rectangle containing a collection of invariant lines, with equivalence  relations imposed among them. Each invariant line corresponds to a periodic orbit.
\end{lemma}
\begin{proof}
	When a family of invariant tori terminates at $B$, a torus transitions into an invariant set that consists of lower-dimensional invariant subsets and open invariant regions between them. Since we are dealing with a non-vanishing vector field, no invariant set of dimension $0$ exists. Consequently, a broken torus consists of invariant circles, each corresponding to a periodic orbit, and open annuli bounded by these circles.
	
	Clearly, each open annulus of $B$ arises as the limit of a family of invariant tori. Its two boundary circles correspond to the limit of two families of closed curves on the tori. The two families of closed curves must belong to the same homology class on the torus. Otherwise, they would intersect, leading to a contradiction.
	
	Without loss of generality, we assume these invariant lines are vertical lines in rectangle coordinates. Starting with an invariant vertical line, we encounter an invariant open annulus, followed by the other boundary of the annulus (which is also an invariant vertical line), then another annulus, and so on. This process continues until all the invariant open annuli are accounted for exactly once. This is feasible because $B$ corresponds to a directed graph where all vertices have even degrees. As a result, we obtain a representation of the broken torus in rectangle coordinates. Naturally, equivalence relations are required between these vertical lines, as the boundaries of different invariant open annuli may overlap.
\end{proof}

\begin{lemma}\label{lem:twofamily}
	For any broken torus $B$, there exists two sets of families of invariant tori $\{T^{-\epsilon}_{1},\dots T^{-\epsilon}_{k}\}$ and $\{T^{\epsilon}_{1},\dots T^{\epsilon}_{l}\}$ with $\epsilon>0$ such that 
	\begin{equation}\label{eq:twofamily}
		B=\bigcup_{i=1}^k\lim_{\epsilon\rightarrow 0}T^{-\epsilon}_i=\bigcup_{j=1}^l\lim_{\epsilon\rightarrow 0}T^{\epsilon}_j.
	\end{equation}
\end{lemma}
\begin{proof}
	In the neighbourhood of $B$, there are different families of invariant tori. Each open annulus in $B$ is part of the limit set of a family of invariant tori and can be approached from both sides. Suppose, for contradiction, that both sides belong to the same family of tori. Then the flow directions on the two sides of the annulus would be opposite, implying that the annulus itself consists of a continuous family of periodic orbits. This contradicts the definition of a regular broken torus. Consequently, each open annulus must be the common boundary of two distinct families of invariant tori. This naturally divides these families of invariant tori into two classes: "inside" and "outside".
\end{proof}

In a rectangle representation of $B$, one may attempt to define the type of a closed curve in $B$ in the usual way, by ignoring the equivalence relations between invariant lines. However, due to these equivalence relations, a closed curve in $B$ might not remain closed in the rectangle representation. Fortunately, this closeness can be ensured when the closed curve is an intersection curve with a Birkhoff section. 

\begin{lemma}\label{lem:homologyonbro}
	Let $\Gamma_0=B\cap S$ be a transverse curve on broken torus. Then the homology type of $\Gamma_0$ is well-defined. $\Gamma_0$ is of $(p_0,q_0)$-type if and only if $\mathbf{Ins}(\Gamma_0,v_b)=p_0$ and  $\mathbf{Ins}(\Gamma_0,h_b)=q_0$, where $v_b$ and $h_b$ are any vertical and horizon lines, respectively, in rectangle representation of $B$.
\end{lemma}
\begin{proof}
	Assume $B$ is a broken torus with invariant vertical lines. Choose a rectangle representation of $B$ such that both coordinates increase along the flow. It suffices to show that $\Gamma_0$ is a closed curve as in an ordinary rectangle without additional equivalence relations between vertical lines.
	
	Consider an open annulus $\mathbb{A}$ with two boundaries given by two vertical lines $v_1$ and $v_2$. Since $v_1$ is transverse to $S$, in neighbourhood of the intersection point, $S\cap \mathbb{A}$ is a transverse curve with a boundary to be the intersection point in $v_1$. If the other boundary of $S\cap \mathbb{A}$ also lies in $v_1$, this would create both a positive and a negative intersection in $v_1$, which is a contradiction. Thus the other boundary point of $S\cap \mathbb{A}$ lies in $v_2$. 
	
	In conclusion, $\Gamma_0$ is composed of a intersection point $v_1\cap S$ followed by a transverse curve in an open annulus, then another intersection point $v_2\cap S$ and another transverse curve in the next open annulus, and so on. This process ensures that $\Gamma$ traverses all the open annuli and invariant vertical lines, ultimately forming a closed curve, just as in the standard case. Thus we can define the homology type of $\Gamma_0$ in the usual way. Specifically, $\Gamma_0$ is of $(p_0,q_0)$-type if and only if $\mathbf{Ins}(\Gamma_0,v_b)=p_0$ and  $\mathbf{Ins}(\Gamma_0,h_b)=q_0$.
\end{proof}

On $T=S^1\times S^1$, there are  two natural generators of the first homology, corresponding to the two $S^1$ factors. The ideal situation is that natural generators of different families coincide on a broken torus. In such a case, we would have a consistent framework for defining the homology type of closed curves across different families of tori and on the broken torus itself. 

However, it is possible for discrepancies to arise. For instance, consider a scenario where a family of $(1,1)$-type curve converges to an invariant line on $B$ from the "inside", while another family of $(1,0)$-type converges to the same invariant line from the "outside". Then there is a non-trivial correspondence between the natural generators of different families of invariant tori. Fortunately, this does not present an obstacle to our construction.

Let $h_b$ and $v_b$ be a horizon line and a vertical line in the rectangle representation. Then $h_b$ and $v_b$ actually defines two generators of every nearby family of tori, with a correspondence to the natural generators. Every transverse curve on an invariant torus in neighbourhood of $B$ can be uniquely represented as a combination of $h_b$ and $v_b$. Therefore, when defining the homology type of closed curves in neighbourhood of $B$, we can use $h_b$ and $v_b$ as the two generators. 

\textbf{For simplicity, we additionally assume that the natural generators in different families of invariant tori coincide. Thus there is consistent way to define the homology of transverse curves on each invariant torus.}

Under the same notations in Lemma \ref{lem:twofamily}, let $\Gamma_i^{-\epsilon}=S\cap T_i^{-\epsilon}$ and $\Gamma_j^{\epsilon}=S\cap T_j^{\epsilon}$. For each family of tori, the homology type of intersection curve is constant. Let $(p_i^-,q_i^-)$ and $(p_j^+,q_j^+)$ be the homology type of $\Gamma_i^{-\epsilon}$ and $\Gamma_j^\epsilon$, respectively. We have the following proposition.

\begin{proposition}\label{pro:homology}
	Let $\Gamma_0=B\cap S$ be a transverse curve of $(p_0,q_0)$-type. If $B$ is a broken torus with invariant vertical lines, then we have
	\begin{equation}\label{eq:pijp0}
		p_i^-=p_j^+=p_0\ne 0, \quad \sum_{i=1}^k q_i^-=\sum_{j=1}^{l}q_j^+=q_0,
	\end{equation}
	where $i\in \{1,\dots k\}$ and $j\in \{1,\dots l\}$.

	Similarly, if $B$ is a broken torus with invariant horizon lines, then we have
	\begin{equation}\label{eq:qijq0}
		\sum_{i=1}^k p_i^-=\sum_{j=1}^{l}p_j^+=p_0, \quad 	q_i^-=q_j^+=q_0\ne 0.
	\end{equation}
\end{proposition}

\begin{proof}
	The proof for two cases is similar. Here, we only prove the case when $B$ has invariant vertical lines. We choose a rectangle representation of $B$ such that both coordinates increase along the flow. 
	
	By Lemma \ref{lem:twofamily},
	\[\Gamma_0=\bigcup_{i=1}^k\lim_{\epsilon\rightarrow 0}\Gamma^{-\epsilon}_i=\bigcup_{i=j}^l\lim_{\epsilon\rightarrow 0}\Gamma^{\epsilon}_j.\]
		
	Let $v_1$ and $v_2$ be two invariant vertical lines that serves as the boundaries of an open annulus(it's possible that $v_1=v_2$). Recall that an open annulus of $B$ is part of the limiting set given by a family of tori. Without loss of generality, let's assume it to be $T_1^\epsilon$ from the previously mentioned families. And $v_1,v_2$ is the limit of two families of vertical lines on the tori. By our assumption, the intersection number of $\Gamma_1^\epsilon$ with any vertical line is the constant $p_1^+$. Taking the limit, we have 
	\[p_0=\mathbf{Ins}(\Gamma_0,v_1)=\mathbf{Ins}(\Gamma_0,v_2)=p_1^+.\]
	Since $\Gamma_0$ is transverse to the flow, $p_0$ can not be $0$. Note that $v_1$ and $v_2$ can be any adjacent invariant vertical lines. We have prove the first part of \eqref{eq:pijp0}.
	
	Choose a broken horizon line $h_B\subset B$ such that the intersections of $\Gamma_0$ and $h_B$ are away from invariant vertical lines. By \eqref{eq:twofamily}, there exist two sets of families of horizon lines $h_i^{-\epsilon}\subset T_i^{-\epsilon}$ such that 
	\[h_B=\lim_{\epsilon\rightarrow 0}\bigcup_{i=1}^k h_i^{-\epsilon}=\lim_{\epsilon\rightarrow 0}\bigcup_{j=1}^l h_j^{\epsilon}.\]
	By definition, $q_i^-=\mathbf{Ins}(h_i^{-\epsilon},\Gamma_i^{-\epsilon})$, $q_j^+=\mathbf{Ins}(h_j^{\epsilon},\Gamma_j^{\epsilon})$ and $q_0=\mathbf{Ins}(h_B,\Gamma_0)$. By taking limit, we have
	\[q_0=\sum_{i=1}^k q_i^-=\sum_{i=j}^l q_j^+.\]
	This complete the proof of the proposition.
	
\end{proof}

\section{Transverse curves near the boundary orbits}\label{sec:4}
When a family of transverse curves in $S$ converges to a curve which fails to remain transversal, the non-transversal part must lie in $\partial S$. In this section, we focus on studying the behavior of transverse curves near the boundary orbits. 

If the family of invariant tori terminates at a periodic orbit, it may not be possible to extend the family of transverse curves continuously to the periodic orbit while maintaining transversality. In such cases, the periodic orbit must be a boundary orbit. In rectangle coordinates, when approaching the periodic orbit, a pair of opposite sides of the rectangle are progressively compressed until they collapse into a single line. The torus then transforms into an invariant vertical or horizontal line. After continuously extending the transverse curve to the periodic orbit, the following lemma becomes evident.
\begin{lemma}\label{lem:endoftori}
	Let $T_\epsilon(\epsilon>0)$ be a family of tori that ends in a periodic orbit as $\epsilon\rightarrow 0$. Suppose the family of transverse curves on $T_\epsilon$ is of $(p,q)$-type. 
	
	If $T_\epsilon(\epsilon>0)$ ends in an invariant vertical line, the family of transverse curves can be extended transversely to the periodic orbits if and only if $q=0$. Otherwise, the boundary curve is a $|q|$-fold covering of the periodic orbit.
	
	Similarly, if $T_\epsilon(\epsilon>0)$ ends in an invariant horizon line, the family of transverse curves can be extended transversely to the periodic orbits if and only if $p=0$. Otherwise, the boundary curve is a $|p|$-fold covering of the periodic orbit.
\end{lemma}

\begin{proposition}\label{pro:necefordisk}
	Let $T$ be an invariant torus such that $\partial S\cap T=\emptyset$ and $\Gamma=S\cap T$ is of $(p,q)$-type. Denote $M^+$ and $M^-$ to be the two connected components of $M\backslash T$. If both $p,q$ are non-zero, then $\partial S\cap M^+\ne \emptyset$ and $\partial S\cap M^-\ne \emptyset$.
\end{proposition}

\begin{proof}
	If $\partial S\cap M^+= \emptyset$, then $S$ must be transverse to all tori in $M^+$. As the invariant tori vary, the homology type of transverse curve remains unchanged until a broken torus appears. By Proposition \ref{pro:homology}, whenever passing through a broken torus, at least one family of tori persists, with the homology type of its transverse curves having both components non-zero. 
	
	Eventually, all families of tori terminate in periodic orbits after passing through finitely many broken tori. By Lemma \ref{lem:endoftori}, the family of tori with both non-zero homology components for the transverse curves must end at a boundary orbit of $S$. This contradicts to our assumption $\partial S\cap M^+= \emptyset$. Thus $\partial S\cap M^+\ne \emptyset$.
	
	Similarly, $\partial S\cap M^-\ne \emptyset$. Therefore, $S$ have at least two boundary orbits, i.e there is no disk-like Birkhoff section.
\end{proof}
\begin{corollary}\label{cor:necefordisk}
	Assume $S$ is a Birkhoff section with a unique boundary orbit and $\Gamma=S\cap T$ is a $(p,q)$-type transverse curve on the torus $T$. Then either $p=0$ or $q=0$.
\end{corollary}

If a Birkhoff section is transverse to all invariant tori and broken tori, then its boundary orbits must be some periodic orbits as the ends of families of tori. Proposition \ref{pro:homology} and Lemma \ref{lem:endoftori} provide sufficient tools for constructing Birkhoff sections of this type.

However, a Birkhoff section may fail to be transverse to an invariant torus. In such cases, there must be at least one boundary orbit on this invariant torus(or broken torus), implying that the invariant torus is rational torus filled with periodic orbits. Notably, all other flow lines, except for the boundary orbits, must intersect the surface transversely. Consequently, the intersection of the Birkhoff section with a non-transversal invariant torus consists of both boundary orbits and transverse curves. 

Clearly, in the neighborhood of a non-transversal invariant torus, there are two families of tori: one approaching from the "inside" and the other from the "outside". By continuity, the limit set of the transverse curves on nearby invariant tori is contained in the intersection of the Birkhoff section with the non-transversal invariant torus. However, the homology type of the transverse curves on the nearby invariant tori cannot be uniquely determined solely from the intersection curve on the non-transversal invariant torus.

The situation turns out to be quite complex due to the following two reasons: (i) Although the union of the limit sets of the two families of transverse curves contains the intersection curve on the non-transversal invariant torus, the precise nature of the limit set for each family of transverse curves remains unknown; (ii) The boundary orbit may be multi-covered, and the multiplicity of this covering is also indeterminate. 

In Fig.\ref{fig:threetypes}, three types of transverse curves near a non-transversal invariant torus are illustrated, where it is assumed that flow lines are given by vertical lines. After taking the limit, part of the transverse curve in Fig.\ref{fig:multi-covered} will converge to a $4$-fold covering of the boundary orbit. In Fig.\ref{fig:twoboundary}, part of the transverse curve will converge to two different boundary orbits. In Fig.\ref{fig:removable}, part of the transverse curve will converge to 
a boundary orbit. This case is called removable because the transverse curve can be deformed into a horizontal line, thereby eliminating this boundary orbit.

\begin{figure}
	\begin{subfigure}[t]{0.32\textwidth} 
		\includegraphics[width=\linewidth]{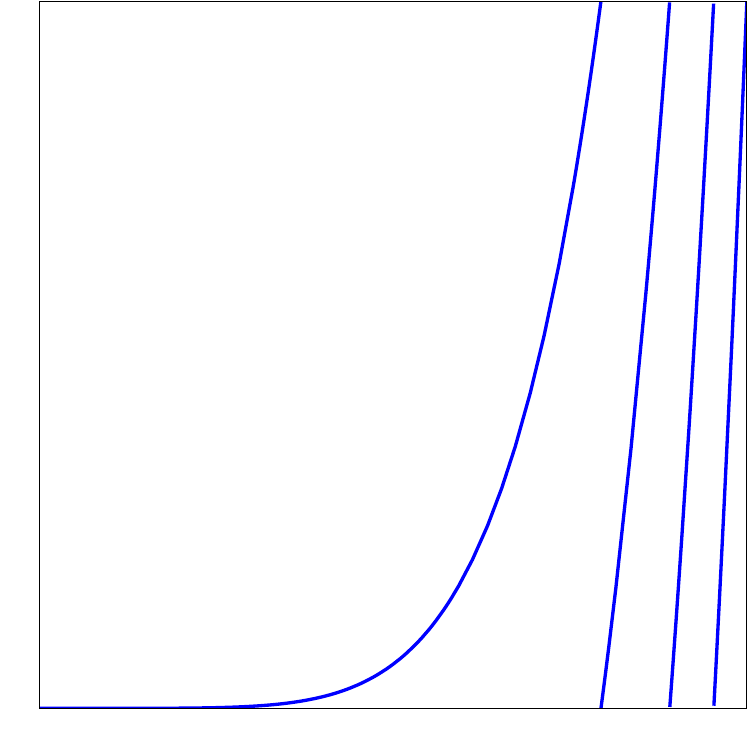}
		\caption{Multi-covered boundary orbit.}
		\label{fig:multi-covered}
	\end{subfigure}
	\hfill
	\begin{subfigure}[t]{0.32\textwidth} 
		\includegraphics[width=\linewidth]{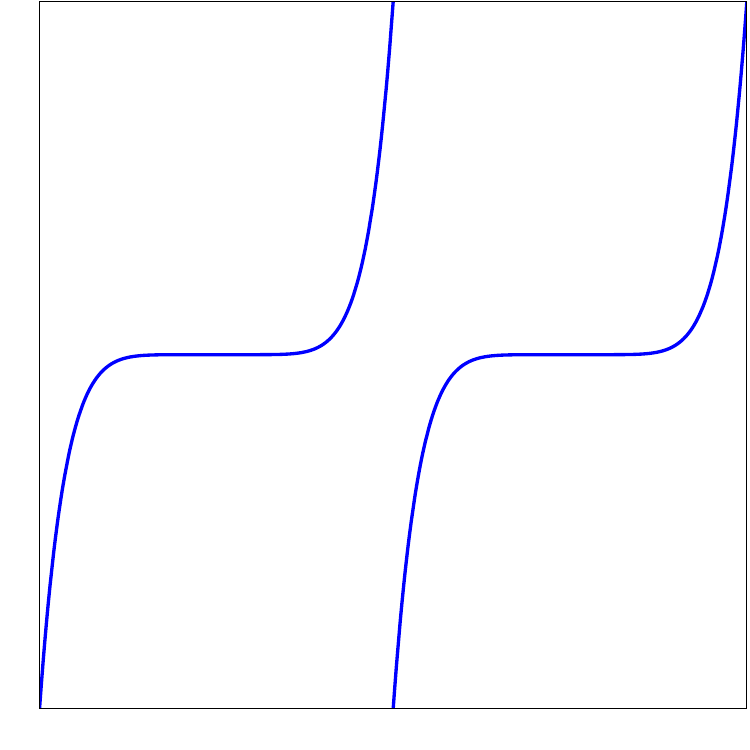}
		\caption{Two boundary orbits on an invariant torus.}
		\label{fig:twoboundary}
	\end{subfigure}
	\hfill
	\begin{subfigure}[t]{0.32\textwidth} 
		\includegraphics[width=\linewidth]{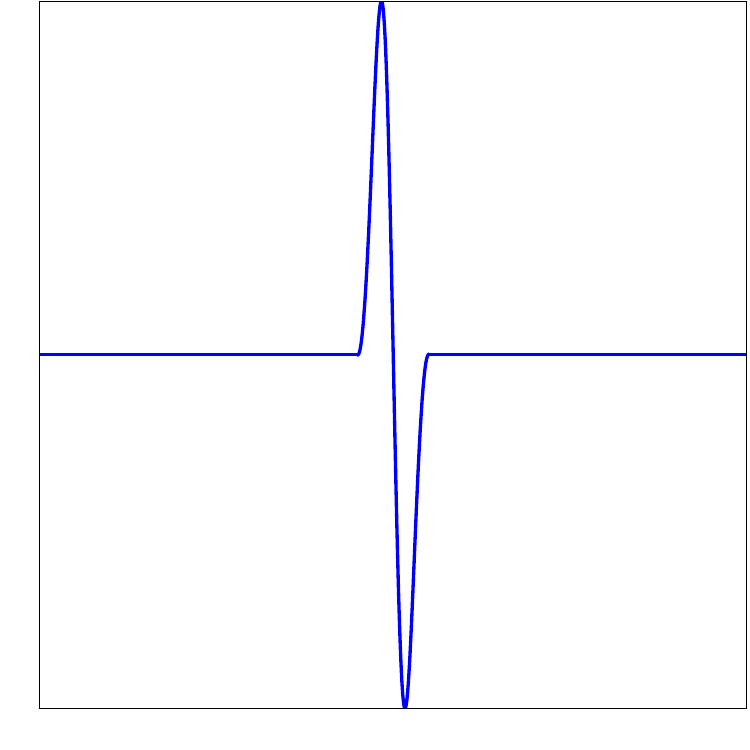}
		\caption{A removable boundary orbit.}
		\label{fig:removable}
	\end{subfigure}
	\caption{Two types of transverse curves near a non-transversal invariant torus.}
	\label{fig:threetypes}
\end{figure}

\textbf{For convenience, we introduce a new rectangle coordinates on the non-transversal torus, which is different from the one used to define homology type of closed curves:} Suppose we have a non-transversal torus with $(P,Q)$-type periodic orbits. By cutting the torus along a boundary orbit, we obtain an annulus where the two boundary circles correspond to the same periodic orbit. We regard this annulus as a rectangle with the top and bottom edges are identified, so that the vertical lines within the rectangle correspond to the periodic orbits on the torus. The left and right vertical edges represent the boundary orbit along which the torus was cut. We also assume that points on the left and right edges, at the same height, correspond to the same point on the periodic orbit. In this representation, a horizon line will form a closed curve.

In Fig.\ref{fig:newcoordinate}, a fundamental domain for the new rectangle coordinates is showed. The blue line segment corresponds to the boundary orbit of $(3,2)$-type, i.e. a vertical line at boundary. The red line segment corresponds to a horizon line, which is transverse to the flow and of $(1,1)$-type. It should be noted that there are different ways to define the fundamental domain for the new rectangle coordinates, as the homology of a horizon line can differ by multiples of $(P,Q)$.

Assume that a fundamental domain for the new rectangle coordinates has been chosen. Let $[v]$ and $[h]$ be the homology type of a vertical line and a horizon line, respectively. Clearly, we have $[v]=(P,Q)$. 
\begin{lemma}\label{lem:decomposition}
	Any $(p,q)$ can be uniquely expressed as a linear combination of $[v]$ and $[h]$. Moreover,  the coefficient corresponding to $[h]$ is independent of the choice of coordinates. 
\end{lemma}

\begin{figure}

	\centering
		\includegraphics[width=0.5\linewidth]{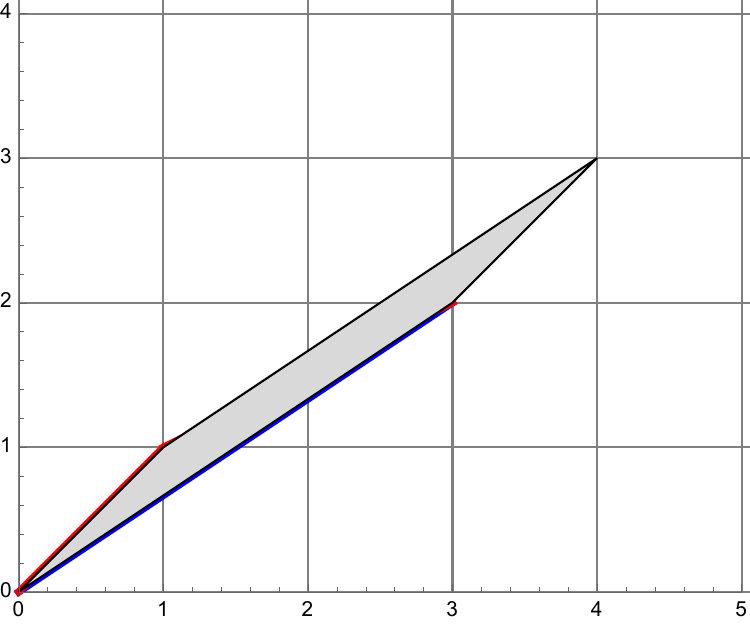}

	\caption{Fundamental domain for new rectangle coordinates with $P=3,Q=2$.}
	\label{fig:newcoordinate}
\end{figure}

Fig.\ref{fig:threetypes}  shows that different choices of transverse curves within the same homology type may result in different types of boundary orbits, which seems to be troublesome. Usually, we aim for the Birkhoff section to be as simple as possible. Therefore, it is preferable to eliminate as many unnecessary boundary orbits as possible.

By Corollary \ref{cor:homotopy}, it is possible to deform the transverse curves on nearby tori such that their limiting curve on the non-transversal invariant torus exhibit well-behaved properties.
\begin{lemma}\label{lem:limitcurve}
	For a family of transverse curves which converging to a non-transversal torus, the transverse curves can be deformed such that both coordinates vary monotonically along the curves in the new rectangle coordinates and there is a single boundary orbit on the non-transversal torus. Moreover, the limiting curve is some horizon lines and the boundary orbit(possibly multi-covered).
\end{lemma}

\begin{proposition}\label{pro:nontranstorus}
	Let $T_s,s\in (-\epsilon,+\epsilon)$ be a family of invariant tori such that $(p^+,q^+)$ and $(p^-,q^-)$ are the corresponding homology type of transverse curves on $T_s$ for $s>0$ and $s<0$, respectively. Then $\partial S\cap T_0\ne \emptyset$ if and only if $(p^+,q^+)$ is not a multiple of $(P,Q)$ and $(p^+,q^+)-(p^-,q^-)=k(P,Q)$ for some $k\ne 0$.
\end{proposition}
\begin{proof}
	The sufficient part follows directly from Lemma \ref{lem:invarianthomology}.We now prove the necessary part.
	
	We shall use the new rectangle coordinates. By Lemma \ref{lem:limitcurve}, $\Gamma_0=T_0\cap S$ consists of some horizon lines and the boundary orbit. By definition of a Birkhoff section, the limiting curve of $\Gamma_s=T_s\cap S$ as $s\rightarrow 0^+$ must go through each horizon line exactly once, with the orientation
	given by positively transverse condition. So is this for $s\rightarrow 0^-$. Let $\alpha $ be the number of horizon lines in $\Gamma_0$.
	
	By Lemma \ref{lem:decomposition}, we have 
	\[(p^+,q^+)=\alpha[h]+\beta^+[v],\]
	\[(p^-,q^-)=\alpha[h]+\beta^-[v].\]
	By Proposition \ref{pro:nopq}, it is necessary that $\alpha\ne 0$. Thus both $(p^+,q^+)$ and $(p^-,q^-)$ are not multiple of $(P,Q)$.
	
If $\beta^+=\beta^-$, then $\Gamma_s$ has the same homology type for $s>0$ and $s<0$. In this case, since  $(p^+,q^+)$ is not multiple of $(P,Q)$, we can remove the boundary orbit by deform $T_s$ into a transverse curve of type $(p^+,q^+)=(p^-,q^-)$ as $s\rightarrow 0$. Thus $\beta^+-\beta^-=k\ne 0$. The necessary part is proved.
	
\end{proof}

We have established the necessary and sufficient conditions for periodic orbits to serve as boundary orbits $\partial S$, either as ends of a family of tori or on an invariant torus. There remains another type of periodic orbit, which serves as an invariant line on a broken tori. 

\begin{proposition}\label{pro:nontransbroken}
	For a broken torus $B$ with vertical invariant lines, $B\cap \partial S= \emptyset$ if and only if 
	\begin{equation}\label{eq:qiqj}
		p_i^-=p_j^+\ne 0, \quad \sum_{i=1}^k q_i^-=\sum_{j=1}^{l}q_j^+.
	\end{equation}

Similarly, if $B$ is a broken torus with invariant horizon lines, then $B\cap \partial S= \emptyset$ if and only if 
\begin{equation}\label{eq:pipj}
	\sum_{i=1}^k p_i^-=\sum_{j=1}^{l}p_j^+, \quad 	q_i^-=q_j^+\ne 0.
\end{equation}
\end{proposition}
\begin{proof}
	The sufficient part is provided by Proposition \ref{pro:homology}. We only need to prove the necessary part. Without loss of generality, assume $B$ has invariant vertical lines. 
	 
	 By Corollary \ref{cor:homotopy}, we can homotope transverse curves on nearby tori such that most parts of the curves are horizontal. The homologies $q_i^-,q_j^+$ are given by these non-horizontal parts. Moreover, we shall make sure that the transverse curves are horizon segments in neighbourhood of invariant vertical lines. Thus non-horizontal parts are in open annuli of $B$. 
	 
	 Note that every open annulus belongs the limit set of two families of open annuli on nearby tori. If the homologies carried by the two families of open annuli are identical, they can be glued together on $B$. Moreover, if the transverse curves in neighbourhood of $B$ can be glued along every open annulus of $B$, then a transverse section is produced in the neighborhood of $B$. 
	 	 
	 Thus, our goal is to distribute the total number $\sum_{i=1}^k q_i^-=\sum_{j=1}^{l}q_j^+$ among the different open annuli of $B$ such that each $q_i^-,q_j^+$ is given by the sum of the numbers in the corresponding open annuli of $B$, as determined by the limit set of the family of tori. We can regard $B$ as a directed graph on $S^2$, where the vertices correspond to invariant lines in $B$. The directed edges represent the open annuli between these invariant lines, with the direction of each edge corresponding to the direction of the flow on these annuli. Each face of this graph corresponds to a family of invariant tori, with an associated number from $q_i^-,q_j^+$ representing the homology of its transverse curves. Thus, the problem reduces to assigning numbers to different edges such that the number associated with each face equals the sum of the numbers on its edges.
	 
	 A family of tori whose limit set contains only one open annulus corresponds to a face with only one edge whose two vertices are identical. It is necessary that the number assigned to such an edge matches the number associated with the face. Without loss of generality, assume $T_1^{-\epsilon}$ converge to a unique open annulus $O_1\subset B$, while the other family of tori is $T_1^{\epsilon}$. Then assign $q_1^-$ to $O_1$, replace $q_1^+$ with $q_1^+-q_1^-$ and remove the open annulus $O_1$. This step reduces the original problem to a similar one with one fewer edge. We repeatedly applying this process until no family of tori remains whose limit set contains only one open annulus.
	 
	 By above argument, we can assume every face have at least two edges. We start with any face that is isomorphic to a closed disk and assign a possible solution for it. Then, we consider the faces adjacent to this initial face (where adjacency means sharing at least one edge). By Lemma \ref{lem:twofamily}, these new faces must belong to the other class. Since the edges of each face form a cycle and the original face is isomorphic to a disk, every newly introduced face must contribute at least one additional edge, and different faces cannot share the same edge. We then assign a valid solution to these new faces. Through this iterative process, the solved region in the graph expands layer by layer. After repeating the process a finite number of times, there will be no remaining edges to assign, and the outermost edge will correspond to the boundary of the final face. The final equation associated with this face will hold automatically due to the condition $\sum_{i=1}^k q_i^-=\sum_{j=1}^{l}q_j^+$. Hence, we have successfully constructed a solution to our problem.
	 
	 Now, different families of tori on both sides of an open annulus converge to the same type of transverse curve on the annulus. It is easy to glue them together to form a transverse section by Corollary \ref{cor:homotopy}. Thus we succeed to make $B$ to be transverse under the assumption $\sum_{i=1}^k q_i^-=\sum_{j=1}^{l}q_j^+$, without changing the homology type of transverse curves on nearby families of tori.
	 	 
\end{proof}
\begin{remark}
	The conclusion of Proposition \ref{pro:nontranstorus} and Proposition \ref{pro:nontransbroken} can not be merged  into a single statement. When $B\cap \partial S\ne \emptyset$, it is possible that $p_i^-=p_j^+=0$, since vertical lines may serve as transverse curves within the open annuli of $B$. 
\end{remark}

By summarising the results in Lemma \ref{lem:endoftori}, Proposition \ref{pro:nontranstorus} and Proposition \ref{pro:nontransbroken}, we have the following theorem.
\begin{theorem}\label{thm:main}
	If $M$ admits a regular toric invariant foliation and the natural generators on different families of tori coincide. Then there is a consistent way to define the homology type of closed curves in $M$. Let $S$ be a Birkhoff section in $M$, we have
	\begin{itemize}
		\item[(1)] If a family of tori ends in an invariant vertical line $P_v$ with transverse curves of $(p,q)$-type. Then $\partial S\cap P_v=0$ if and only if $q=0$. Otherwise, the boundary curve is a $|q|$-fold covering of the periodic orbit;
		\item[(2)] If a family of tori ends in an invariant horizon line $P_h$ with transverse curves of $(p,q)$-type. Then $\partial S\cap P_h=0$ if and only if $p=0$. Otherwise, the boundary curve is a $|p|$-fold covering of the periodic orbit;
		\item[(3)] Let $T_0$ be an invariant tori of $(P,Q)$-type and $(p^+,q^+),(p^-,q^+-)$ are corresponding types of transverse curves in two nearby families of tori. Then $\partial S\cap T_0\ne \emptyset$ if and only if $(p^+,q^+)$ is not a multiple of $(P,Q)$ and $(p^+,q^+)-(p^-,q^-)=k(P,Q)$ for some $k\ne 0$. The boundary curve will be a $|k|$-fold covering of the periodic orbit.
		\item[(4)] If $B$ is a broken torus with invariant vertical lines, then $B\cap \partial S= \emptyset$ if and only if 
		\begin{equation}
			p_i^-=p_j^+\ne 0, \quad \sum_{i=1}^k q_i^-=\sum_{j=1}^{l}q_j^+.
		\end{equation}
	\item[(5)] If $B$ is a broken torus with invariant horizon lines, then $B\cap \partial S= \emptyset$ if and only if 
	\begin{equation}
		\sum_{i=1}^k p_i^-=\sum_{j=1}^{l}p_j^+, \quad 	q_i^-=q_j^+\ne 0.
	\end{equation}
	\end{itemize}

\end{theorem}

\section{Applications}\label{sec:application}
\subsection{Boundary of a toric domain}
\begin{definition}\label{def:toric}
	A \textsl{toric domain} $X_\Omega$ is the preimage of a region $\Omega\subset \mathbb{R}_{\ge 0}^2$
	\[\mu:\mathbb{C}^2\rightarrow \mathbb{R}^2, \quad (z_1,z_2)\rightarrow (\pi|z_1|^2,\pi|z_2|^2).\]
\end{definition}

The \textsl{ellipsoids}
\[E(a,b)=\left\{(z_1,z_2)\Big|\frac{\pi|z_1|^2}{a}+\frac{\pi|z_2|^2}{b}\le 1 \right\},\]
and \textsl{polydisks}
\[P(a,b)=\left\{(z_1,z_2)\Big|\frac{\pi|z_1|^2}{a}\le 1, \frac{\pi|z_2|^2}{b}\le 1 \right\},\]
are examples of toric domains. The $\Omega$ for $E(a,b)$ are right triangles with legs on the axes, while $\Omega$ for polydisks are rectangles with bottom and left sides on the axes. 

In this section, we consider a toric domain where $\Omega$ is a bounded closed set in $\mathbb{R}_{\ge 0}^2$ bounded by a smooth embedded curve $f: [0,1]\rightarrow \mathbb{R}^2$ such that $f(0)=(0,a)$ and $f(1)=(b,0)$. Denote $\partial \Omega$ to be the image of $f$ and $M_f$ to be the three manifold $\mu^{-1}(\partial\Omega)$. As a $3$-dimensional hypersurface in $\mathbb{R}^4$, $M_f$ carries a standard Hamiltonian flow generated by the vector field  $J\cdot \mathbf{n}$, where $J$ is the standard complex structure in $\mathbb{R}^4$ and $\mathbf{n}$ is the outer normal vector to $\partial\Omega$. We define the homology type of closed curves in $M_f$ such that the circle $\mu^{-1}((0,a))$ corresponds to a vertical line of $(0,1)$-type, $\mu^{-1}((b,0))$ corresponds to a horizon line of $(1,0)$-type. 

Depending on $\partial\Omega$, the toric domain $X_\Omega$ can be classified into several types. $X_\Omega$ is a \textsl{monotone toric domain} if $\partial\Omega$ is the graph of a monotone decreasing function. If moreover $\Omega$ is convex, then $X_\Omega$ is a \textsl{convex toric domain}. If $\mathbb{R}_{\ge 0}^2\backslash \Omega$ is convex, then $X_\Omega$ is called a \textsl{concave toric domain}. The following is obvious.
\begin{lemma}
	The boundary of a toric domain is isomorphic to $S^3$ foliated by a single family of tori. The vector field for flow on the invariant torus $\mu^{-1}(f(s))$$(s\in (0,1))$ is parallel to the normal vector at $f(s)$.
\end{lemma}

\begin{lemma}\label{lem:embedded}
	Any $(p,q)$ type piecewise smoothly immersed oriented curve in $T$ can be approximated by an embedded curve of the same type.
\end{lemma}
\begin{proof}
	This is a consequence of standard resolution procedure. Without loss of generality, we assume that the immersed curve is smooth and that all self-intersections are transverse. Let $N$ be a small neighbourhood of an self-intersection point. The immersed curve intersect $\partial N$ at $p_1,p_2,p_3,p_4$, and within $\partial N$, it consists of two oriented segments: one from $p_1$ to $p_3$ and the other from $p_2$ to $p_4$. Clearly, we can approximate the immersed curve within $N$ by two wo non-intersecting curves: one from $p_1$ to $p_4$ and the other from $p_2$ tp $p_3$, in a way that preserves the overall orientation. This does not change the type of curve on torus. We can apply the same procedure to all the self-intersection point. Repeating this process at every self-intersection yields an embedded curve of the same type.
\end{proof}

For a disk-like Birkhoff section, it's necessary that the boundary orbit is an unknot. Note that any periodic orbit on an invariant torus is a torus knot. The following lemma provides the necessary and sufficient condition for a torus knot to be an unknot.
\begin{lemma}
	A torus knot is an unknot if and only if $|p|=1$ or $|q|=1$.
\end{lemma}
If a $(p,q)$-type bounds a Birkhoff section and the corresponding torus knot is not unknotted, then the Birkhoff section must have non-zero genus. In knot theory, the genus of a knot is defined as the minimal genus among all orientable, connected Seifert surfaces whose boundary is the knot. For a $(p,q)$ torus knot, the genus is 
\[g=\frac{(|p|-1)(|q|-1)}{2}.\]

\begin{figure}
	
	\centering
	\includegraphics[width=0.5\linewidth]{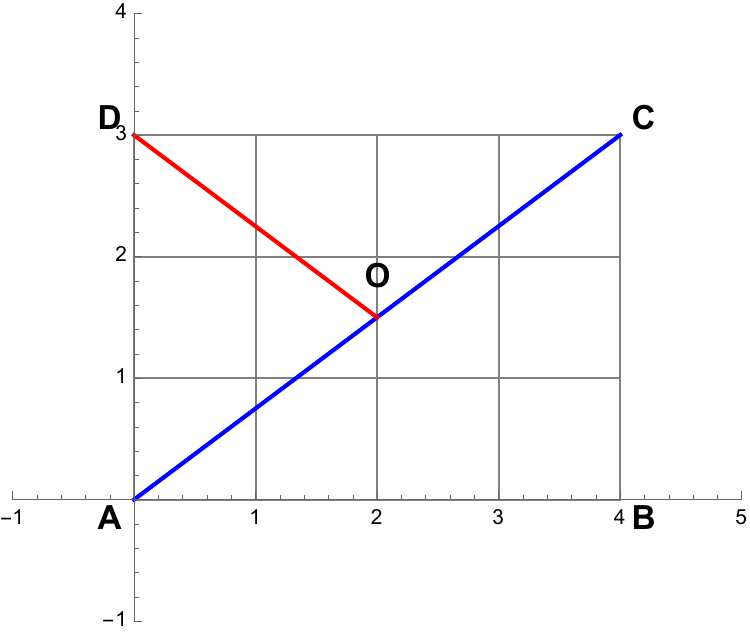}
	
	\caption{The limit curve on non-transversal torus, where the curve is illustrated in $\mathbb{R}^2$ and the torus is represented by $\mathbb{R}^2/\mathbb{Z}^2$.}
	\label{fig:nontrans}
\end{figure}
\begin{theorem}\label{thm:birktoric}
	The standard Hamiltonian flow on $M_f$ has the following properties:
	\begin{itemize}
		\item[(1)] $\mu^{-1}(f(0))$ bounds a disk-like Birkhoff section if and only if $(0,1)$ is not a normal vector at $f(s)$ for all $s\in(0,1)$; 
		\item[(2)] $\mu^{-1}(f(1))$ bounds a disk-like Birkhoff section if and only if $(1,0)$ is not a normal vector at $f(s)$ for all $s\in(0,1)$; 
		\item[(3)] $\mu^{-1}(f(0))$ and $\mu^{-1}(f(1))$ bound an annulus-like Birkhoff section if and only if there exist $p\ne 0$ and $q\ne0$ such that $(p,q)$ is not a normal vector at $f(s)$ for all $s\in(0,1)$;
		\item[(4)] a $(p,q)$-type periodic orbit on a rational invariant torus $\mu^{-1}(f(s_0))(0<s_0<1)$ bounds a Birkhoff section if and only if $(1,0)$ is not a normal vector at $f(s)$ for all $s\in(0,s_0]$ and $(0,1)$ is not a normal vector at $f(s)$ for all $s\in[s_0,1)$. Moreover, the genus of Birkhoff section is $\frac{k(|p|-1)(|q|-1)}{2}$ if the boundary is $k$-fold covered.
	\end{itemize}
\end{theorem}
\begin{proof}
	The cases in (1)-(3) can be proved in the same setting: the family of invariant tori $\mu^{-1}(f(s))$ for $s\in (0,1)$ are transverse to the flow. By Lemma \ref{lem:pqtransversecurve} and Lemma \ref{lem:invarianthomology}, $M_f$ admits a Birkhoff section with $(p,q)$-type transverse curves if and only if $(p,q)$ is not a normal vector to $f(s)$ for all $s\in(0,1)$. Therefore,statements (1)-(3) follow directly from Theorem \ref{thm:main}.
	
	In case (4), the boundary orbit lies on an invariant torus. Assume the periodic orbits on the torus are of $(p,q)$-type. For a Birkhoff section with a unique boundary orbit such that $S\cap T_{s_0}\ne 0$,  by Theorem \ref{thm:main}, it is necessary that $(1,0)$ is not a normal vector to $f(s)$ for all $s\in(0,s_0)$ and $(0,1)$ is not a normal vector to $f(s)$ for all $s\in(s_0,1)$. Moreover, both $(1,0)$ and $(0,1)$ can not be a normal vector at $f(s_0)$ by $(3)$ of Theorem \ref{thm:main}. We must have $S\cap T_s$ are of $(kp,0)$-type for $s\in(0,s_0)$, and of $(0,-kq)$-type for $s\in(s_0,1)$. Then the boundary $\partial S$ will be $|k|$-covering of the periodic orbit.
	
	For the sufficient part, we only prove the conclusion for $k=1$. $\partial S\cap T_{s_0}$ consists of two parts: the periodic orbit(the non-transversal part) and the transversal part. An illustrated for $\partial S\cap T_{s_0}$ is given in Fig. \ref{fig:nontrans} for $(p,q)=(4,3)$, where the blue segment represents the periodic orbit, the red segment is the transversal part and $T_{s_0}=\mathbb{R}^2/\mathbb{Z}^2$. We shall construct $(p,0)$-type transverse curves on $T_s$ for $s\in(0,s_0)$ such that the limiting curve on $T_{s_0}$ is the segments from $D$ to $O$, and then to $C$. This construction follows from \ref{lem:embedded} and the transversality is easy to ensure. Similarly, we can construct $(0,-q)$-type transverse curves on $T_s$ for $s\in(s_0,1)$ whose limit curve on $T_{s_0}$ is the segments from $D$ to $O$, and then to $A$. By Theorem \ref{thm:main}, $S$ is transversal at two boundary circle $\mu^{-1}(f(0))$ and $\mu^{-1}(f(1))$. Thus $S$ is a  Birkhoff section with boundary to be a $(p,q)$-type periodic orbit on $T_{s_0}$.
	
	To compute the genus of the Birkhoff section, it suffices to compute the Euler characteristic. In \cite{Dehornoy2022}, Dehornoy and Rechtman derived a formula for computing the Euler characteristic of a transverse surface in a homology sphere, expressed in terms of the linking numbers and self-linking numbers of its boundary orbits. Our setting is a simplified case of their general result. Since we construct the Birkhoff section explicitly, we present an elementary approach instead.
	
	If $X=X_1\cup X_2$ and $Y=X_1\cap X_2$, the Euler characteristic of $X$  is given by
	\begin{equation}
		\chi(X)=\chi(X_1)+\chi(X_2)-\chi(Y).
	\end{equation}
	
	Denote $S^-$ to be the part of $S$ in the solid torus $\mu^{-1}(f(s))(s\in [0,s_0])$ and $S^+$ to be the part in $\mu^{-1}(f(s))(s\in [s_0,1])$. Then $S^-\cap S^+$ is the transversal segment on $T_{s_0}$. Thus 
	$$\chi(S)=\chi(S^-)+\chi(S^+)-1.$$
	By our construction, $S^-$ is composed of $p$ disks intersecting at some points and $S^+$ is composed of $q$ disks intersecting at some points. These point are exactly the intersection points of the transversal segment with the boundary orbit(i.e. the intersection points of blue segment and red segment) with the two boundary points be excluded. Thus total multiplicity of these points is $pq-1$. We have
	\[\chi(S^-)+\chi(S^+)=p+q-pq+1.\]
	And 
	\[\chi(S)=p+q-pq.\]
	Thus, the genus of $S$ is 
	\begin{equation}
		g(s)=\frac{1-\chi(S)}{2}=\frac{(p-1)(q-1)}{2}.
	\end{equation}
	which is exactly the genus of boundary orbit.

\end{proof}

Note that the Birkhoff section becomes a global surface of section if $\partial S$ is embedded. There is a version of Theorem \ref{thm:birktoric} for global surface of section, with only a minor modification.
\begin{theorem}\label{thm:globaltoric}
	The standard Hamiltonian flow on $M_f$ has the following properties:
	\begin{itemize}
		\item[(1)] $\mu^{-1}(f(0))$ bounds a disk-like global surface of section if and only if $(0,1)$ is not a normal vector at $f(s)$ for all $s\in(0,1)$; 
		\item[(2)] $\mu^{-1}(f(1))$ bounds a disk-like global surface of section if and only if $(1,0)$ is not a normal vector at $f(s)$ for all $s\in(0,1)$; 
		\item[(3)] $\mu^{-1}(f(0))$ and $\mu^{-1}(f(1))$ bound an annulus-like global surface of section if and only if there exist $|p|=|q|=1$ such that $(p,q)$ is not a normal vector at $f(s)$ for all $s\in(0,1)$;
		\item[(4)] a $(p,q)$-type periodic orbit on a rational invariant torus $\mu^{-1}(f(s_0))(0<s_0<1)$ bounds a global surface of section if and only if $(1,0)$ is not a normal vector at $f(s)$ for all $s\in(0,s_0]$ and $(0,1)$ is not a normal vector at $f(s)$ for all $s\in[s_0,1)$. Moreover, the genus is $\frac{(p-1)(q-1)}{2}$.
	\end{itemize}
\end{theorem}

\begin{corollary}
	For the flow on the boundary of a weakly convex toric domiain, any  $(p,q)$-type periodic orbit with positive $p,q$ bounds a global surface of section with genus $\frac{(p-1)(q-1)}{2}$.
\end{corollary}

The theorems above consider Birkhoff sections in $M_f$ with a single boundary orbit. By Proposition \ref{pro:boundary_exist}, one can easily construct examples such that any Birkhoff section for the flow must have multiple boundary orbits.

\begin{lemma}[Lemma $3$ in \cite{Hofer1993}]\label{lem:hofer1993}
	The standard Hamiltonian flow on a star-shape hypersurface in $\mathbb{R}^4$ is equivalent to a Reeb flow on the tight $S^3$.
\end{lemma}
\begin{theorem}[Theorem 1.1 in \cite{vanKoert2022}]\label{thm:nonsimple}
    For any $k>0$ an integer, there exist a star-shaped $M_f$ whose flow is equivalent to a Reeb flow on $S^3$ such that any Birkhoff section for the flow must has at least $k$ distinct boundary orbits.
\end{theorem}
\begin{proof}
	Consider an isosceles triangle with its vertex at the origin. There is a unique semicircle whose diameter is the base of the triangle, forming a star-shaped curve. By Proposition \ref{pro:boundary_exist}, if the boundary curve of a toric domain contains such a semicircle, then any Birkhoff section for the flow will have a boundary orbit contained in the semicircle part. 
	
	Now the star-shape curve $f:[0,1]\rightarrow \mathbb{R}^2$ is constructed as following. Take $k$ non-overlapping isosceles triangles in the first quadrant, each with its vertex at the origin. Then each isosceles triangle defines a unique star-shape semicircle such that the base is diameter of the semi-circle. Let $f$ be a star-shape curve containing these $k$ different semicircles. Then any Birkhoff section of $M_f$ must have at least $k$ boundary orbits. 
	
	Note that the boundary of a toric domain is star-shape in $\mathbb{R}^4$ if and only if $\partial \Omega$ is star-shape in $\mathbb{R}^2$. By Lemma \ref{lem:hofer1993}, the flow on $M_f$ is equivalent to a Reeb flow on tight $S^3$. The proof is complete.
\end{proof}
This gives a different example from that in \cite{vanKoert2020} for Reeb flows on tight $S^3$ without a simple global surface of section. 

Combining the result of Edtair \cite{Edtmair2024}, which states that the cylindrical capacity of a dynamically convex domain in $\mathbb{R}^4$ agrees with the least symplectic area of a disk-like global surface of section of the Reeb flow on the boundary of the domain. We can reprove the following result of Gutt and Hucthings.
\begin{theorem}[Theorem 1.7 in \cite{Hutchings2022}]
	If $X_\Omega$ is a monotone toric domain in $\mathbb{R}^4$, then $c_B(X_\Omega)=c_Z(X_\Omega)$.
\end{theorem}
\begin{proof}
	By an approximation argument, we can assume that $X_\Omega$ is strictly monotone(i.e. both components of the out normal vector at at every boundary point are positive). It is shown in \cite{Hutchings2022} that strictly monotone is equivalent to dynamically convex for toric domains in $\mathbb{R}^4$. Denote $\Delta(a)$ to be the triangle with vertices $(0,0),(a,0)$ and $(0,a)$. The toric domain corresponding to $\Delta(a)$ is the standard ball $B(a):=\Big\{z\in \mathbb{R}^4\Big| \pi|z|^2\le a\Big\}$. 
	
	For a strictly monotone toric domain $X_\Omega$, let $\Delta(r)$ be the largest triangle in $X_\Omega$. Clearly, we have $r\le c_B(X_\Omega)$.
	We assume the boundary of $X_\Omega$ is a smooth curve $f:[0,1]\rightarrow \mathbb{R}^2$ as before, and $\Delta(r)$ touch the curve at some point $f(s_0)$. Then $f(s_0)$ is either a point on axis or a point on the curve with $(1,1)$ to be a out normal vector. In either case, by Theorem \ref{thm:globaltoric}, the periodic orbit on $\mu^{-1}(f(s_0))$ bounds a disk-like global surface of section. It easy to verify that the symplectic area of this global surface of section is $r$. By Theorem 1.9 in \cite{Edtmair2024}, we have $c_Z(X_\Omega)= r$. The proof is complete.
	
\end{proof}

\subsection{Energy hypersurface of separable Hamiltonian}
A Hamiltonian is said to be separable if the Hamiltonian function can written as the sum of several independent two-dimensional Hamiltonian functions, i.e. $H=\sum_{i=1}^{n}H_i(x_i,y_i)$. In this section, we'll consider the case when $n=2$ and every regular energy level set of $H$ will be a three manifold with a non-vanishing Hamiltonian vector field. It is obvious that every regular energy level set of separable Hamiltonian $H=H_1(x_1,y_1)+H_2(x_2,y_2)$ admits a toric invariant foliation. Thus our results can be applied to study Birkhoff sections for such flows. 

For simplicity, we assume that: (i) both $H_1$ and $H_2$ has a unique global minimum point with the value to be $0$; (ii) all critical points of $H_1$ and $H_2$ are isolated. 

Let $H=c$ be a regular level set, then the corresponding three manifold is foliated by the product spaces $\{H_1=s\}\times \{H_2=c-s\}$ with $s\in [0,c]$. If both $H_1$ and $H_2$ take regular value, then the product will be a torus(or a union of tori, depending on the topology of the level sets). If either $H_1$ or $H_2$ takes a critical value, the product will be a broken torus or a periodic orbit.

It should be noted that the boundary of a toric domain is not generally a level set of a separable Hamiltonian. It is obvious that the boundary of a monotone toric domain can be also regard as the level set of a separable Hamiltonian. However, if the toric domain is not monotone, it may be not possible to be realize its boundary as such a level set. Even when this is possible, $H_1$ or $H_2$ must contain a continuous set of critical points, which is degenerate.

To maintain consistency, for a torus in the product $\{H_1=s\}\times \{H_2=c-s\}$, we shall let $H_1=s$ and $H_2=c-s$ correspond to horizon line and vertical line, respectively. With this convention, the following lemma is immediate.
\begin{lemma}\label{lem:separable}
	\begin{itemize}
		\item [(i)] For every invariant torus, we define the flow lines(with orientation) of $H_1$ and $H_2$ to be the generators of first homology of the torus. This provides a consistent way to define homology types of closed curves on all invariant tori;
		\item [(ii)] On each invariant torus, the vector of linear motion is in the first closed quadrant;
		\item [(iii)] Any saddle point(neither local minimum or maximum point) of $H_1$ corresponds to a broken torus with invariant vertical lines. Any saddle point of $H_2$ corresponds to a broken torus with invariant horizon lines.
	\end{itemize}
\end{lemma}

Note that we do not require the critical points to be non-degenerate. With a slight abuse of terminology, we will refer to saddle points as \textsl{hyperbolic critical points} and to maxima or minima as \textsl{elliptic critical points} in the following discussion.

Since the critical points of $H_1$ and $H_2$ are isolated. There are only finitely many families of tori for regular value $H=c$. By (ii) of Lemma \ref{lem:separable}, every $(p,q)$ with $p,q$ having opposite signs can be a homology type of transverse curves in each family of tori. By applying Theorem \ref{thm:main}, we can construct varies types of Birkhoff sections in $H=c$ by Theorem \ref{thm:main} such that the boundary of the Birkhoff section corresponds to local maximum or local minimum points of $H_1$ or $H_2$.
\begin{theorem}
	The flow on any regular level set of $H$ admits a Birkhoff section.
\end{theorem}

Clearly, the topology of the level set $H=c$ depends on $c$, which can be derived from classical Morse theory. Similarly, the possible Birkhoff sections in $H=c$ also depend on the critical points of $H_1$ and $H_2$. Here we give a criterion for $H=c$ to admits a Birkhoff section with a unique boundary orbit.

\begin{theorem}\label{thm:oneboundaryforH}
	Assume $H=c$ is a regular level set. Then the following are equivalent.
	\begin{itemize}
		\item[(i)] $H=c$ admits a disk-like global surface of section;
		\item[(ii)] $H=c$ admits a Birkhoff section with a unique boundary orbit;
		\item[(iii)]  Both sublevel sets $H_1^c,H_2^c$ are isomorphic to a disk. And there exist a $c_0\in (0,c)$ such that there is no critical point for $H_1>c_0$ and $H_2>c-c_0$.
	\end{itemize}
\end{theorem}
\begin{remark}
	$(iii)$ implies that the level set $H=c$ is isomorphic to $S^3$. 
\end{remark}

Note that the two ends of every family of tori correspond to critical points of $H_1$ or $H_2$. We introduce the following definition for convenience.
\begin{definition}
	If two ends of a family of tori $T_s$ correspond to critical points of $H_1$ and $H_2$, respectively. we say that $T_s$ $connects$ $H_1$ and $H_2$. More specifically, we can say $T_s$ connects an elliptic critical points of $H_i$ and a hyperbolic critical points of $H_j$, where $i,j\in {1,2}$.
\end{definition}
If $H=c$ is a connected set, then all critical points of $H_1$ and $H_2$ can be connected through finitely many families of tori. $H_1=0$ and $H_2=0$ are two necessary critical points, thus there is at least one family of tori connecting $H_1$ and $H_2$. Clearly, if $T_s$ connect an elliptic critical points of $H_1$ and an elliptic critical points of $H_2$, then the closure of $T_s$ is a connected component of $H_c$, which is isomorphic to $S^3$.

\begin{lemma}\label{lem:nohyperboundary}
	The periodic orbit corresponding to a hyperbolic critical point of $H_1$ or $H_2$ can not be the unique boundary orbit of a Birkhoff section.
\end{lemma}
\begin{proof}
	Assume the unique boundary orbit corresponds to a hyperbolic critical point of $H_1$, which is an invariant vertical line on a broken torus. By Corollary \ref{cor:necefordisk} and Theorem \ref{thm:main}, the first component of homology type of transverse curves on nearby tori must vanish. On the other hand, for transverse curves near other invariant vertical lines(i.e. other critical points of $H_1$), the second component of homology type must vanish. By Theorem \ref{thm:main}, a $(p,0)$-type transverse curve can not transition to a $(0,q)$-type transverse curve while maintaining the transversality. This is a contradiction.
\end{proof}

\begin{lemma}\label{lem:uniquefamily}
	If $H=c$ admits a Birkhoff section with a unique boundary orbit, then there is a unique family of tori connecting $H_1$ and $H_2$.
\end{lemma}
\begin{proof}
	Clearly, $H=c$ must be connected. Note that for an elliptic critical point, there is only one family of tori connecting it. If $T_s$ is a family of tori connecting an elliptic critical points of $H_1$ and an elliptic critical points of $H_2$. Then closure of $T_s$ is $S^3$, $T_s$ actually connects $H_1=0$ and $H_2=0$ and there is no other critical points. The proof is trivial in this case. Therefore, we assume that there is at least one hyperbolic critical point in the two limit sets of $T_s$.
	
	By Corollary \ref{cor:necefordisk}, if all tori of $T_s$ are transverse, then the homology type is $(p,0)$ or $(0,q)$. Since $T_s$ connects an invariant vertical line and an invariant horizon line. By Theorem \ref{thm:main}, there has to be a boundary orbit on two ends. By Lemma \ref{lem:nohyperboundary}, $T_s$ must connect an elliptic critical point and a hyperbolic critical point, with the elliptic point corresponding to a boundary orbit.
	
	If not all tori in $T_s$ are transverse, then there exists a boundary orbit on a rational torus. 
	
	We have proven that each family of tori connecting $H_1$ and $H_2$ produces a boundary orbit. The boundary orbit is either on a rational torus or corresponds to an elliptic critical point. Since different family of tori can not converge to the same elliptic point, the boundary orbits associated with different families of tori connecting $H_1$ and $H_2$ must be distinct. Recall our assumption that $H=c$ admits a Birkhoff section with a unique boundary orbit, there must be a unique family of tori connecting $H_1$ and $H_2$.
	
\end{proof}

\begin{lemma}\label{lem:c}
	If there is only one family of tori connecting $H_1$ and $H_2$, then both sublevel sets $H_1^c,H_2^c$ are isomorphic to a disk.Furthermore, there exist a $c_0\in (0,c)$ such that there is no critical points for $H_1>c_0$ and $H_2>c-c_0$.
\end{lemma}
\begin{proof}
	We first proof that sublevel set $H_1^c:=\{x\in\mathbb{R}^2| H_1(x)\le c\}$ is isomorphic to a disk. If $H_1=c$ has more than one component. Denote $l_1$ and $l_2$ to be two different component of $H_1=c$. By decreasing the value of $H_1$, we get two families of closed curves starting from $l_1$ and $l_2$, respectively. We continue this process until the family of curves touch a critical point of $H_1$, which is necessarily a hyperbolic critical point. Note that $l_1,l_2$ correspond to the critical point $H_2=0$. As $l_1$ approaches its corresponding hyperbolic critical point, we get some families of tori, connected end-to-end, linking the two critical points. This implies a family of tori connecting $H_1$ and $H_2$. So is it for $l_2$. Since the two families of curves starting from $l_1$ and $l_2$ are disjoint, we have two distinct families of tori connecting $H_1$ and $H_2$. This contradicts to Lemma \ref{lem:uniquefamily}. Thus $H_1=c$ has only one component, and $H_1^c$ is isomorphic a disk.
	
	For existence of $c_0$, we also prove by contradiction. Let $a,b$ be two critical points corresponding to largest critical values of $H_1,H_2$, respectively.  If no such $c_0$ exists, then $H_1(a)+H_2(b)>c$. By decreasing along level sets of $H_1$ starting from $a$, as the argument in last paragraph, we get a family of tori connecting $H_1$ and $H_2$. Also, there is a family of tori connecting $H_1$ to $H_2$ obtained by decreasing the level set of $H_1$ from $H_1=c$ to a hyperbolic critical point. Now we have two families of tori connecting $H_1$ and $H_2$. Since their projection to $H_1^c$ are disjoint, the two families are distinct. Again, we get a contradiction. The proof is complete.
\end{proof}
\textsl{Proof of Theorem \ref{thm:oneboundaryforH}:}	\textbf{(iii)$\Rightarrow$(ii)}
	 By assumption, the set $\{H_1=c_0\}\times \{H_2=c-c_0\}$ is a single invariant torus. Let $T_s$ be the family containing this torus, then $T_s$ is the unique family of invariant tori connecting $H_1$ and $H_2$. Any invariant torus in $T_s$ divides the energy level set $H=c$ into two parts: one part contains only critical points of $H_1$, and the other part contains only critical points of $H_2$. 
	 
	 If there is a $(p,q)$ rational torus in $T_s$, following the proof of Theorem \ref{thm:birktoric}, we can construct a Birkhoff section whose boundary orbit is a $(p,q)$ periodic orbit on the torus. For the tori in the part containing critical points of $H_1$, the transverse curves have homology type $(p,0)$. For the tori in the other part, the transverse curves have homology type $(0,-q)$.
	
	If all tori in $T_s$ are irrational, it is necessary that the flow on all tori are the same irrational flow and both $H_1$ and $H_2$ have no critical points in $(0,c)$. $H=c$ must be the boundary of an irrational ellipsoid. Thus, either $H_1=0$ or $H_2=0$ bounds a disk-like global surface of section. 
	
	\textbf{(ii)$\Rightarrow$(iii)} This is a consequence of Lemma \ref{lem:uniquefamily} and Lemma \ref{lem:c}.
	
	\textbf{(i)$\Leftrightarrow$(ii)} Clearly, $(i)$ implies $(ii)$. It remains to prove that $(ii)$ implies $(i)$. By Lemma \ref{lem:uniquefamily}, there is a unique family of tori $T_s$ connecting $H_1$ and $H_2$. If it connects an elliptic critical point of $H_1$ and an elliptic critical points of $H_2$, the proof is trivial. Thus we assume there is a hyperbolic critical point in limit set of $T_s$. Clearly, the torus $\{H_1=c_0\}\times \{H_2=c-c_0\}$ belongs to $T_s$. As argued above, every $(p,q)$ periodic orbit in $T_s$ bounds a global surface of section with genus $\frac{(p-1)(q-1)}{2}$. Near the hyperbolic critical point, the linear flow on the nearby invariant tori becomes nearly horizontal or vertical. We can find a $(p,q)$ rational torus with $p$ or $q$ to be $1$. The corresponding global surface of section then has genus $0$, which is a disk-like global surface of section.
\qed

At the end of this section, we present a simple example to illustrate the meaning of Theorem \ref{thm:oneboundaryforH}.

\begin{example}\label{exm:5.10}
	Let $H_1=H_2=\frac{1}{2}x^2+g(y)$. Then $g(y)$ determines the critical points of $H_1$ and $H_2$. Consider the case where $g(y)$ has three critical points, as illustrated in Fig.\ref{fig:g_1(y)}: two local minimum at values $0, 0.2$ and a local maximum at $0.3$.
	
	Consider the energy level set $H=H_1+H_2=c$. 
	\begin{itemize}
		\item For $c\in (0,0.2)$, $H=c$ is isomorphic to $S^3$ with only one family of invariant tori. Either $H_1=0$ or $H_2=0$ bounds a disk-like global surface of section;
		\item For $c\in (0.3,0.4)$, one can verify that $H=c$ is isomorphic to $S^3$. The condition in $(iii)$ of Theorem \ref{thm:oneboundaryforH} does not holds. Hence, no disk-like global surface of section exists;
		\item For $c> 0.6$, $H=c$ is still isomorphic to $S^3$. By Theorem \ref{thm:oneboundaryforH}, it admits disk-like global surface of section.
	\end{itemize}

\end{example}

\begin{figure}
	
	\centering
	\includegraphics[width=0.5\linewidth]{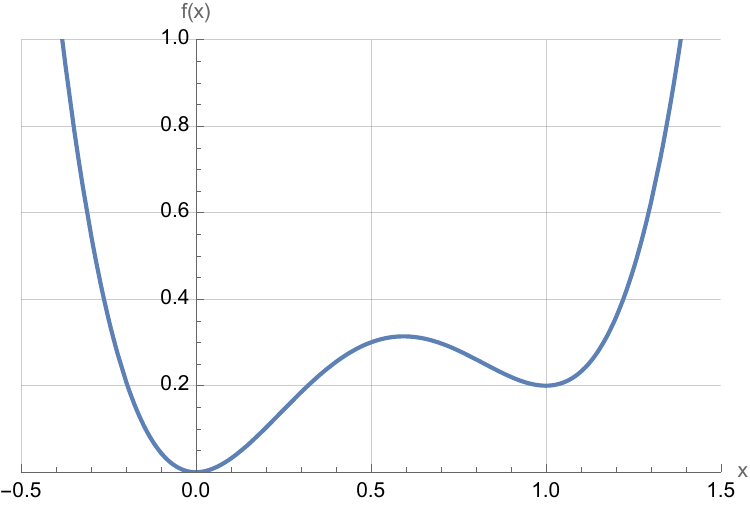}
	
	\caption{Graph of $g(y)\le 1$.}
	\label{fig:g_1(y)}
\end{figure}

\textbf{Acknowledgements:} The authors are grateful to P.A.S. Salom\~{a}o for pointing out an error in the first draft of this article. They also thank O. van Koert, Jun Zhang, and Jianfeng Lin for valuable discussions.

\bibliographystyle{abbrv}
\bibliography{reference.bib}

\begin{thebibliography}{10}

\bibitem{Birkhoff1966}
G.~D. Birkhoff.
\newblock {\em Dynamical systems}, volume Vol. IX of {\em American Mathematical
  Society Colloquium Publications}.
\newblock American Mathematical Society, Providence, RI, 1966.
\newblock With an addendum by Jurgen Moser.

\bibitem{VPUA2024}
V.~Colin, P.~Dehornoy, U.~Hryniewicz, and A.~Rechtman.
\newblock Generic properties of {$3$}-dimensional {R}eeb flows: {B}irkhoff
  sections and entropy.
\newblock {\em Comment. Math. Helv.}, 99(3):557--611, 2024.

\bibitem{Colin2023}
V.~Colin, P.~Dehornoy, and A.~Rechtman.
\newblock On the existence of supporting broken book decompositions for contact
  forms in dimension 3.
\newblock {\em Invent. Math.}, 231(3):1489--1539, 2023.

\bibitem{CM2022}
G.~Contreras and M.~Mazzucchelli.
\newblock Existence of {B}irkhoff sections for {K}upka-{S}male {R}eeb flows of
  closed contact 3-manifolds.
\newblock {\em Geom. Funct. Anal.}, 32(5):951--979, 2022.

\bibitem{Dan2019}
D.~Cristofaro-Gardiner, M.~Hutchings, and D.~Pomerleano.
\newblock Torsion contact forms in three dimensions have two or infinitely many
  {R}eeb orbits.
\newblock {\em Geom. Topol.}, 23(7):3601--3645, 2019.

\bibitem{Liu2024}
D.~Cristofaro-Gardiner, M.~Hutchings, H.~Umberto, Leone, and L.~Hui.
\newblock Proof of hofer-wysocki-zehnder's two or infinity conjecture.
\newblock 2024.

\bibitem{ZhangJun2024}
J.~Dardennes, J.~Gutt, and J.~Zhang.
\newblock Symplectic non-convexity of toric domains.
\newblock {\em Commun. Contemp. Math.}, 26(4):Paper No. 2350010, 21, 2024.

\bibitem{Pedro2022}
N.~V. de~Paulo and P.~A.~S. Salom\~ao.
\newblock Reeb flows, pseudo-holomorphic curves and transverse foliations.
\newblock {\em S\~ao Paulo J. Math. Sci.}, 16(1):314--339, 2022.

\bibitem{Dehornoy2022}
P.~Dehornoy and A.~Rechtman.
\newblock Vector fields and genus in dimension 3.
\newblock {\em Int. Math. Res. Not. IMRN}, (5):3262--3277, 2022.

\bibitem{Edtmair2024}
O.~Edtmair.
\newblock Disk-like surfaces of section and symplectic capacities.
\newblock {\em Geom. Funct. Anal.}, 34(5):1399--1459, 2024.

\bibitem{Franks1992}
J.~Franks.
\newblock Geodesics on {$S^2$} and periodic points of annulus homeomorphisms.
\newblock {\em Invent. Math.}, 108(2):403--418, 1992.

\bibitem{Fried1983}
D.~Fried.
\newblock Transitive {A}nosov flows and pseudo-{A}nosov maps.
\newblock {\em Topology}, 22(3):299--303, 1983.

\bibitem{Giroux2002}
E.~Giroux.
\newblock G\'eom\'etrie de contact: de la dimension trois vers les dimensions
  sup\'erieures.
\newblock In {\em Proceedings of the {I}nternational {C}ongress of
  {M}athematicians, {V}ol. {II} ({B}eijing, 2002)}, pages 405--414. Higher Ed.
  Press, Beijing, 2002.

\bibitem{Hutchings2022}
J.~Gutt, M.~Hutchings, and V.~G.~B. Ramos.
\newblock Examples around the strong {V}iterbo conjecture.
\newblock {\em J. Fixed Point Theory Appl.}, 24(2):Paper No. 41, 22, 2022.

\bibitem{HO2024}
P.~Haim-Kislev and Y.~Ostrover.
\newblock A counterexample to viterbo's conjecture.
\newblock 2024.

\bibitem{Hofer1993}
H.~Hofer.
\newblock Pseudoholomorphic curves in symplectizations with applications to the
  {W}einstein conjecture in dimension three.
\newblock {\em Invent. Math.}, 114(3):515--563, 1993.

\bibitem{Hofer1998}
H.~Hofer, K.~Wysocki, and E.~Zehnder.
\newblock The dynamics on three-dimensional strictly convex energy surfaces.
\newblock {\em Ann. of Math. (2)}, 148(1):197--289, 1998.

\bibitem{Hofer2003}
H.~Hofer, K.~Wysocki, and E.~Zehnder.
\newblock Finite energy foliations of tight three-spheres and {H}amiltonian
  dynamics.
\newblock {\em Ann. of Math. (2)}, 157(1):125--255, 2003.

\bibitem{HS2011}
U.~Hryniewicz and P.~A.~S. Salom\~ao.
\newblock On the existence of disk-like global sections for {R}eeb flows on the
  tight 3-sphere.
\newblock {\em Duke Math. J.}, 160(3):415--465, 2011.

\bibitem{Hryniewicz2014}
U.~L. Hryniewicz.
\newblock Systems of global surfaces of section for dynamically convex {R}eeb
  flows on the 3-sphere.
\newblock {\em J. Symplectic Geom.}, 12(4):791--862, 2014.

\bibitem{HSW2023}
U.~L. Hryniewicz, P.~A.~S. Salom\~ao, and K.~Wysocki.
\newblock Genus zero global surfaces of section for {R}eeb flows and a result
  of {B}irkhoff.
\newblock {\em J. Eur. Math. Soc. (JEMS)}, 25(9):3365--3451, 2023.

\bibitem{Hutchings2014}
M.~Hutchings.
\newblock Lecture notes on embedded contact homology.
\newblock In {\em Contact and symplectic topology}, volume~26 of {\em Bolyai
  Soc. Math. Stud.}, pages 389--484. J\'anos Bolyai Math. Soc., Budapest, 2014.

\bibitem{vanKoert2022}
J.~Kim, Y.~Kim, and O.~van Koert.
\newblock Reeb flows without simple global surfaces of section.
\newblock {\em Involve}, 15(5):813--842, 2022.

\bibitem{Kuperburg1996}
G.~Kuperberg.
\newblock A volume-preserving counterexample to the {S}eifert conjecture.
\newblock {\em Comment. Math. Helv.}, 71(1):70--97, 1996.

\bibitem{Taubes2007}
C.~H. Taubes.
\newblock The {S}eiberg-{W}itten equations and the {W}einstein conjecture.
\newblock {\em Geom. Topol.}, 11:2117--2202, 2007.

\bibitem{vanKoert2020}
O.~van Koert.
\newblock A {R}eeb flow on the three-sphere without a disk-like global surface
  of section.
\newblock {\em Qual. Theory Dyn. Syst.}, 19(1):Paper No. 36, 16, 2020.

\end{thebibliography}

\end{document}